\documentclass[reqno,12pt]{amsart}
\usepackage{amsfonts}
\usepackage{bbm}
\usepackage{}
\numberwithin{equation}{section}
\usepackage{color}

\usepackage{indentfirst}
\usepackage{color}
\usepackage{amssymb}
\usepackage{mathrsfs}
\usepackage{xy}
\xyoption{all}
\def\Ext{\mbox{\rm Ext}\,} \def\Hom{\mbox{\rm Hom}} \def\dim{\mbox{\rm dim}} \def\Iso{\mbox{\rm Iso}}
\def\lr#1{\langle #1\rangle}    
   \def\im{\mbox{\rm Im}\,} 

\def\zsum{\sum\limits_{i\in\mathbb{Z}_m}}\def\osum{\bigoplus\limits_{i\in\mathbb{Z}_m}}\def\zprod{\prod\limits_{i\in\mathbb{Z}_m}}

\def\Aut{\mbox{\rm Aut}}\def\A{\mathcal{A}\,} \def\E{\mathcal{E}\,}

\theoremstyle{plain}
\newtheorem{theorem}{\bf Theorem}[section]
\newtheorem{lemma}[theorem]{\bf Lemma}
\newtheorem{corollary}[theorem]{\bf Corollary}
\newtheorem{proposition}[theorem]{\bf Proposition}

\theoremstyle{definition}
\newtheorem{definition}[theorem]{\bf Definition}
\newtheorem{remark}[theorem]{\bf Remark}
\newtheorem{example}[theorem]{\bf Example}

\newcommand{\bt}{\begin{theorem}}
\newcommand{\et}{\end{theorem}}
\newcommand{\bl}{\begin{lemma}}
\newcommand{\el}{\end{lemma}}
\newcommand{\bd}{\begin{definition}}
\newcommand{\ed}{\end{definition}}
\newcommand{\bc}{\begin{corollary}}
\newcommand{\ec}{\end{corollary}}
\newcommand{\bp}{\begin{proof}}
\newcommand{\ep}{\end{proof}}
\newcommand{\bx}{\begin{example}}
\newcommand{\ex}{\end{example}}
\newcommand{\br}{\begin{remark}}
\newcommand{\er}{\end{remark}}
\newcommand{\be}{\begin{equation}}
\newcommand{\ee}{\end{equation}}
\newcommand{\ba}{\begin{align}}
\newcommand{\ea}{\end{align}}
\newcommand{\bn}{\begin{enumerate}}
\newcommand{\en}{\end{enumerate}}
\newcommand{\bcs}{\begin{cases}}
\newcommand{\ecs}{\end{cases}}

\makeatletter
\renewcommand{\section}{\@startsection{section}{1}{0mm}
  {-\baselineskip}{0.5\baselineskip}{\bf\leftline}}
\makeatother

\begin{document}

\title[Periodic derived Hall algebras of hereditary abelian categories]{Periodic derived Hall algebras of\\ hereditary abelian categories}
\author{Haicheng Zhang}
\address{Institute of Mathematics, School of Mathematical Sciences, Nanjing Normal University,
 Nanjing 210023, P. R. China.\endgraf}
\email{zhanghc@njnu.edu.cn}


\begin{abstract}
Let $m$ be a positive integer and $D_m(\mathcal {A})$ be the $m$-periodic derived category of a finitary hereditary abelian category $\A$. Applying the derived Hall numbers of the bounded derived category $D^b(\mathcal {A})$, we define an $m$-periodic extended derived Hall algebra for $D_m(\mathcal {A})$, and use it to give a global, unified and explicit characterization for the algebra structure of Bridgeland's Hall algebra of periodic complexes. Moreover, we also provide an explicit characterization for the odd periodic derived Hall algebra of $\A$ defined by Xu-Chen \cite{XuChen}.
\end{abstract}

\maketitle

\section{Introduction}
Hall algebras provide a good framework for additive categorification of Lie algebras, quantum groups and (quantum) cluster algebras. Ringel \cite{R90} used the Hall algebra of a representation-finite hereditary algebra to give a realization of the positive part of the corresponding quantum group. Green \cite{Gr95} generalized this result to Kac--Moody type, and gave a bialgebra structure on the Ringel--Hall algebra of hereditary algebras, whose antipode was provided by Xiao \cite{Xiao}. Hence, the extended Ringel--Hall algebra of hereditary algebras is a Hopf algebra and
Xiao \cite{Xiao} obtained a realization of the whole quantum group by constructing the Drinfeld double of the extended Ringel--Hall algebras.

In order to give an intrinsic realization of the entire quantum group via Hall algebra approach, one tried to define a Hall algebra for triangulated categories, in particular, for root categories.
Kapranov \cite{Kap} defined an associative algebra called lattice
algebra for the bounded derived category of a hereditary
abelian category. This algebra provides a Heisenberg double of the Borel part of the quantum group, which is closely related to the whole quantum group but does
not coincide with it.
In 2006, To\"en \cite{Toen} defined a Hall algebra, called derived Hall algebra, for a differential graded category satisfying some finiteness conditions.
Later on, Xiao and Xu \cite{XiaoXu} generalised To\"en's construction to any triangulated category satisfying certain homological finiteness conditions.
Unfortunately, the finiteness conditions one imposes on a triangulated category in order to define its derived
Hall algebra are quite restrictive: they do hold for bounded derived categories of finitary abelian categories, but they are not satisfied for periodic triangulated
categories. Therefore, none of these techniques can give a satisfactory construction of the whole quantum group as a Hall algebra of some kind.

In 2011, Xu and Chen \cite{XuChen} revised the construction in \cite{XiaoXu} and defined the Hall algebra of odd periodic triangulated categories. It is a pity that their proof approach can not be applied to define the Hall algebra of even periodic triangulated categories. One still expects to give a realization of the whole quantum group via a certain Hall algebra of root categories.

Instead of the Hall algebra of root categories, Bridgeland \cite{Br} considered the Ringel--Hall algebra of the category of $2$-periodic complexes of projective modules over a hereditary algebra $A$. By taking some localizations and reductions with respect to contractible complexes, he obtained a realization of the entire quantum group via such localized Hall algebra, called Bridgeland's Hall algebra of $A$, which was proved by Yanagida \cite{Yan} to be isomorphic to the Drinfeld double Hall algebra of $A$.

Let $\A$ be a finitary hereditary abelian category. Recently, by analysing the relations between the extensions in the bounded derived category $D^b(\A)$ and the root category $\mathcal {R}(\A)$, the author \cite{Zhang} defined a Hall algebra for the root category $\mathcal {R}(\A)$
by applying the (dual) derived Hall numbers of $D^b(\A)$, and proved that it is isomorphic to the Drinfeld double Hall algebra of $\A$, and thus it can be used to provide a global realization of the corresponding quantum group.

In this paper, let $m$ be a positive integer and $\A$ be a finitary hereditary abelian category. We generalise the constructions in \cite{Zhang} to the $m$-periodic derived category $D_m(\A)$. Explicitly, we define an $m$-periodic extended derived Hall algebra for $D_m(\mathcal {A})$ by applying the (dual) derived Hall numbers of the bounded derived category $D^b(\mathcal {A})$, and prove that it is isomorphic to Bridgeland's Hall algebra of $m$-periodic complexes associated to $\A$.
Moreover, we also introduce an $m$-periodic derived Hall algebra for odd periodic derived categories, which is proved to be isomorphic to the odd periodic derived Hall algebra of $\A$ defined by Xu-Chen \cite{XuChen}. Compared with $m$-periodic extended derived Hall algebras, the odd periodic derived Hall algebras can be defined without appending the $K$-elements in the basis elements.

The paper is organized as follows: we recall the definitions and properties of Ringel--Hall algebra and derived Hall algebra of a hereditary abelian category $\A$ in Section 2. In Section 3, we define an extended derived Hall algebra for the $m$-periodic derived category $D_m(\A)$ and prove its associativity. Section 4 is devoted to defining a Hall algebra without appending the $K$-elements in the basis elements for the odd periodic derived categories, and comparing it with the odd periodic derived Hall algebra given by Xu-Chen \cite{XuChen}. In Section 5, we prove that the $m$-periodic extended derived Hall algebra is isomorphic to  Bridgeland's Hall algebra of $m$-periodic complexes.

Throughout the paper, $m$ is a positive integer, $\mathbb{F}_q$ is a finite field with $q$ elements and $v=\sqrt{q}\in\mathbb{C}$, $\A$ is an essentially small hereditary abelian $\mathbb{F}_q$-linear category, and $D^b(\mathcal {A})$ is the bounded derived category of $\A$ with the shift functor $[1]$. We always assume that $\A$ is finitary, i.e., for any objects $M,N\in\A$, the spaces $\Hom_{\A}(M,N)$ and $\Ext^1_{\A}(M,N)$ are both finite dimensional. Let $K(\mathcal{A})$ be the Grothendieck group of $\A$, we denote by $\hat{M}$ the image of $M$ in $K(\mathcal{A})$ for any $M\in\A$. For a finite set $S$, we denote by $|S|$ its cardinality. For an essentially small finitary category $\mathcal {E}$, we denote by ${\rm Iso}(\mathcal {E})$ the set of isomorphism classes $[X]$ of objects $X$ in $\mathcal {E}$; for each object $X\in\mathcal {E}$, denote by $\Aut_\mathcal {E} (X)$ the automorphism group of $X$. We denote the quotient ring $\mathbb{Z}/m\mathbb{Z}$ by $\mathbb{Z}_m=\{0,1,\ldots,m-1\}$, and set $\prod\limits_{i\in\mathbb{Z}_m}a_i:=
a_0 a_1\cdots a_{m-1}$. Let $G$ be a group acting on a set $X$, the set consisting of $G$-orbits is denoted by $\frac{X}{G}$.

\section{Preliminaries}
In this section, we recall the definitions and properties of the Ringel--Hall algebra, derived Hall algebra of $\A$.
\subsection{Ringel--Hall algebras}
For any objects $L,M,N \in \mathcal{A}$, we denote by $\Ext_\mathcal{A}^1(M,N)_L$ the subset of $\Ext_\mathcal{A}^1(M,N)$, which is consisting of the equivalence classes of short exact sequences with middle term $L$.
\begin{definition}
The \emph{Hall algebra} $\mathcal {H}(\mathcal{A})$ of $\mathcal{A}$ is the $\mathbb{C}$-vector space with the basis $\{u_{[M]}~|~[M]\in \Iso(\mathcal{A}$)\}, and with the multiplication defined by
\begin{equation}\label{jhallm}
u_{[M]} \diamond u_{[N]} = \sum\limits_{[L] \in {\rm Iso}(\mathcal{A})} {\frac{{|\Ext_\mathcal{A}^1{{(M,N)}_L}|}}{{|\Hom_\mathcal{A}(M,N)|}}}u_{[L]}.\end{equation}
\end{definition}
By \cite{R90a, Br}, the above operation $\diamond$ defines on $\mathcal {H}(\mathcal{A})$ the structure of a unital associative algebra, and the basis element $u_{[0]}$ is the unit.
\begin{remark}
For any objects $L,M,N\in \A$, let
$g_{M,N}^L$ be the number of subobjects $X$ of $L$ such that $X\cong N$ and $L/X\cong M$.
The Riedtmann--Peng formula (cf. \cite{Riedtmann,Peng}) states that
$$g_{M,N}^{L}=\frac{|\Ext^1_{\A}(M,N)_{L}|}{|\Hom_{\A}(M,N)|}\cdot \frac{|{\rm Aut}_{\A}(L)|}{|{\rm Aut}_{\A}(M)||{\rm Aut}_{\A}(N)|}.$$ Thus in terms of alternative generators $\mu_{[M]}=\frac{1}{|{\rm Aut}_{\A}(M)|}u_{[M]}$, the product (\ref{jhallm}) takes the form
$$\mu_{[M]}\diamond \mu_{[N]}= \sum\limits_{[L] \in {\rm Iso}(\mathcal{A})}g_{M,N}^L\mu_{[L]},$$
which is the definition used, for example, in \cite{R90a,Sc}.
\end{remark}

For any objects $M,N \in \mathcal{A}$, set $$\lr{M,N}:=\dim_k\Hom_{\A}(M,N)-\dim_k\Ext^1_{\A}(M,N),$$
and it descends to give a bilinear form
$$\lr{\cdot ,\cdot }: K(\mathcal{A})\times K(\mathcal{A})\longrightarrow \mathbb{Z},$$ known as the \emph{Euler form}. We also consider the \emph{symmetric Euler form}
$$(\cdot ,\cdot ): K(\mathcal{A})\times K(\mathcal{A})\longrightarrow \mathbb{Z},$$ defined by $(\alpha,\beta)=\lr{\alpha,\beta}+\lr{\beta,\alpha}$ for all $\alpha,\beta \in K(\mathcal{A})$.
The \emph{Ringel--Hall algebra} ${\mathcal {H}}_{\rm{tw}}(\mathcal{A})$ of $\mathcal{A}$ is the same vector space as $\mathcal {H}(\mathcal{A})$, but with the multiplication defined by $$u_{[M]}u_{[N]}=v^{\lr{\hat{M},\hat{N}}}\cdot u_{[M]}\diamond u_{[N]}.$$
The \emph{extended Ringel--Hall algebra} ${\mathcal {H}}_{\rm{tw}}^{\rm e}(\mathcal{A})$ of $\A$ is defined as an extension of ${\mathcal {H}}_{\rm{tw}}(\mathcal{A})$ by appending elements $K_{\alpha}$ for all $\alpha\in K(\A)$, and imposing relations $$K_{\alpha}K_{\beta}=K_{\alpha+\beta},\quad K_{\alpha}u_{[M]}=v^{(\alpha,\hat{M})}u_{[M]}K_{\alpha},$$ for $\alpha,\beta\in K(\A)$ and $[M]\in \Iso(\mathcal{A})$.

\subsection{Derived Hall algebras}
For any objects $M,N,X\in D^b(\A)$, set $$\{M,N\}:=\prod\limits_{i>0}|\Hom_{D^b(\A)}(M[i],N)|^{(-1)^i}.$$ In a triangulated category $\mathcal {T}$, we denote by $\Hom_{\mathcal {T}}(M,N)_X$ the subset of $\Hom_{\mathcal {T}}(M,N)$ consisting of the morphisms $M\to N$ whose cone is isomorphic to $X$. According to \cite{Toen,XiaoXu}, for any objects $X,Y,L\in D^b(\A)$, we have that
$$\frac{|\Hom_{D^b(\A)}(L,X)_{Y[1]}|}{|{\rm Aut}_{D^b(\A)}(X)|}\cdot\frac{\{L,X\}}{\{X,X\}}=
\frac{|\Hom_{D^b(\A)}(Y,L)_{X}|}{|{\rm Aut}_{D^b(\A)}(Y)|}\cdot\frac{\{Y,L\}}{\{Y,Y\}}=:F_{X,Y}^L,$$ which is called \emph{To\"en's formula}. It is easy to see that for any $X,Y,L\in\A$ we have that $F_{X,Y}^L=g_{X,Y}^L$.

\begin{definition}
The \emph{derived Hall algebra} $\mathcal {D}\mathcal {H}(\A)$ of $\A$ is the $\mathbb{C}$-vector space with the basis $\{\mu_{[X]}~|~[X]\in {\rm Iso}(D^b(\A))\}$, and with the multiplication defined by
\begin{equation}\label{dhallm}
\mu_{[X]} \mu_{[Y]}=\sum\limits_{[L]\in {\rm Iso}(D^b(\A))}F_{X,Y}^L\mu_{[L]}.\end{equation}
\end{definition}

By \cite{Toen,XiaoXu}, we know that $\mathcal {D}\mathcal {H}(\A)$ is an associative and unital algebra. The derived Riedtmann--Peng formula (cf. \cite{XiaoXu2,WWZ}) states that
\begin{equation}\label{drpgs1}
F_{X,Y}^L=\frac{|\Ext^1_{D^b(\A)}(X,Y)_{L}|}{|\Hom_{D^b(\A)}(X,Y)|}\cdot\frac{1}{\{X,Y\}}\cdot\frac{|{\rm Aut}_{D^b(\A)}(L)|}{|{\rm Aut}_{D^b(\A)}(X)||{\rm Aut}_{D^b(\A)}(Y)|}\cdot\frac{\{L,L\}}{\{X,X\}\{Y,Y\}},
\end{equation}
where $\Ext^1_{D^b(\A)}(X,Y)_{L}:=\Hom_{D^b(\A)}(X,Y[1])_{L[1]}$. Thus in terms of alternative generators $u_{[X]}=\{X,X\}\cdot |{\rm Aut}_{D^b(\A)}(X)|\cdot\mu_{[X]}$, the product (\ref{dhallm}) takes the form
$$u_{[X]} u_{[Y]}=\sum\limits_{[L]\in {\rm Iso}(D^b(\A))}H_{X,Y}^Lu_{[L]},$$ where
$$H_{X,Y}^L=\frac{|\Ext^1_{D^b(\A)}(X,Y)_{L}|}{|\Hom_{D^b(\A)}(X,Y)|}\cdot\frac{1}{\{X,Y\}}.$$

For any objects $X_1,X_2,\cdots,X_t,L$ in $D^b(\A)$, define $F_{X_1,X_2,\cdots,X_t}^L$ to be the number such that
$$\mu_{[X_1]}\mu_{[X_2]}\cdots \mu_{[X_t]}=\sum\limits_{[L]\in {\rm Iso}(D^b(\A))}F_{X_1,X_2,\cdots,X_t}^L\mu_{[L]}.$$
By the associativity of the {derived Hall algebra} $\mathcal {D}\mathcal {H}(\A)$, we have that
\begin{equation}\label{associativity}
\begin{split}
F_{X_1,X_2,\cdots,X_t}^L&=\sum\limits_{[X]\in {\rm Iso}(D^b(\A))}F_{X_1,\cdots,X_{i-1},X}^LF_{X_i,\cdots,X_{t}}^X\\
&=\sum\limits_{[X']\in {\rm Iso}(D^b(\A))}F_{X_1,\cdots,X_{i}}^{X'}F_{X',X_{i+1},\cdots,X_{t}}^L\end{split}\end{equation} for each $1< i< t$.

By abuse of notation, in what follows, for each object $X$, we may also write $u_X$ for $u_{[X]}$.
For any $X\in\A$, we set $a_X=|{\rm Aut}_{\A}(X)|$. The following formulas on (derived) Hall numbers are needed in the sequel.


\begin{theorem}{\rm(\textbf{Green's formula},\cite{Gr95})}
Given objects $M,N,M',N'\in\A$, we have the following formula
\begin{equation*}
\begin{split}
&a_Ma_Na_{M'}a_{N'}\sum\limits_{[L]\in{\rm Iso}(\A)}F_{M,N}^LF_{M',N'}^L\frac{1}{a_L}\\&=
\sum\limits_{[A],[A'],[B],[B']\in{\rm Iso}(\A)}q^{-\lr{\hat{A},\hat{B}'}}
F_{A,A'}^{M}F_{B,B'}^{N}F_{A,B}^{M'}F_{A',B'}^{N'}a_{A}a_{A'}a_{B}a_{B'}.
\end{split}
\end{equation*}
\end{theorem}

\begin{lemma}\label{jichuhallshu} $($\cite[Proposition 7.1]{Toen}$)$
For any objects $M,N,X,Y\in\A$, we have that
\begin{equation*}
F_{M[1],N}^{X[1]\oplus Y}=F_{M,N[-1]}^{X\oplus Y[-1]}=q^{-\lr{\hat{Y},\hat{X}}}\frac{a_Xa_Y}{a_Ma_N}\sum\limits_{[L]\in{\rm Iso}(\A)}a_LF_{L,X}^MF_{Y,L}^N.
\end{equation*}
\end{lemma}

\begin{lemma}\label{hallzhuan}$($\cite{Zhang}$)$
For any objects $M, M_1, M_2, I, J$ in $\A$, we have that

$(1)$~\begin{flalign*}H_{I[1]\oplus M_1, M_2\oplus J[-1]}^M&=q^{-\lr{\hat{M}_1,\hat{I}}-\lr{\hat{J},\hat{M}_2}}\frac{a_{M_1}a_{M_2}a_{I}a_{J}}{a_M}F_{M_1\oplus I[1],J[-1]\oplus M_2}^M\\
&=q^{\lr{\hat{J},\hat{I}}-\lr{\hat{M}_1,\hat{I}}-\lr{\hat{J},\hat{M}_2}}\frac{a_{M_1}a_{M_2}a_{I}a_{J}}{a_M}F_{M_1,J[-1],I[1],M_2}^M.\end{flalign*}

$(2)$~$$F_{M_1,J[-1],I[1],M_2}^M=\sum\limits_{[N],[L]\in{\rm Iso}(\A)}\frac{a_Na_L}{a_{M_1}a_{M_2}}F_{J,N}^{M_1}F_{L,I}^{M_{2}}F_{NL}^M.$$

\end{lemma}

\section{$m$-periodic extended derived Hall algebras}
Let $G$ be an automorphism of $D^b(\A)$. We recall that the {\em orbit category} $D^b(\A)/G$ is the category whose objects are
by definition the $G$-orbits $\widetilde{X}$ of objects $X$ in $D^b(\A)$ and morphisms are given by
\begin{equation}\label{guidaodef}
\Hom_{D^b(\A)/G}(\widetilde{X},\widetilde{Y})=\bigoplus\limits_{i\in\mathbb{Z}}\Hom_{D^b(\A)}(X,G^iY).\end{equation}
If $G$ is the $m$-th shift functor $[m]$ of $D^b(\A)$, the corresponding orbit category is called the {\em $m$-periodic derived category} of $\A$, which is also denoted by ${D}_m(\A)$.
According to \cite{PX97,Kellero}, the $m$-periodic derived category ${D}_m(\A)$ is a triangulated category, and its suspension functor is also denoted by $[1]$. In this section, we define a Hall algebra for the $m$-periodic derived category ${D}_m(\A)$ and prove its associativity.

For any triangle in $D_m(\A)$ $$\osum \widetilde{B}_i[i]\longrightarrow \osum \widetilde{M}_i[i]\longrightarrow \osum \widetilde{A}_i[i]\longrightarrow \osum \widetilde{B}_i[i+1]$$ such that $A_i,B_i,M_i\in\A$ with $i\in\mathbb{Z}_m$, by considering the homologies,
we have the following cyclic exact sequence in $\A$
\begin{equation}\label{cyclicexact}
{\cdots\xrightarrow{}B_0\xrightarrow{} M_0\xrightarrow{} A_0\xrightarrow{f_{m-1}} B_{m-1}\xrightarrow{} M_{m-1}\xrightarrow{}\cdots
\xrightarrow{f_1}B_1\xrightarrow{}M_1\xrightarrow{}A_1\xrightarrow{f_0}B_0\xrightarrow{}\cdots}.\end{equation}
Thus, for each $i\in\mathbb{Z}_m$, taking $I_i=\im f_i$, we obtain the exact sequences in $\A$
\begin{equation}\label{fiveterm}
\xymatrix{0\ar[r]&I_i\ar[r]&B_i\ar[r]&M_i\ar[r]&A_i\ar[r]&I_{i-1}\ar[r]&0.}
\end{equation}
By \cite[Lemma 3.1]{Zhang}, for each $i\in\mathbb{Z}_m$, we obtain the following triangle in $D^b(\A)$
$$\xymatrix{B_i\oplus I_{i-1}[-1]\ar[r]&M_i\ar[r]&I_i[1]\oplus A_i.}$$
Hence, we have the dual derived Hall numbers $H_{I_i[1]\oplus A_i,B_i\oplus I_{i-1}[-1]}^{M_i}$ for each $i\in\mathbb{Z}_m$.
Inspired by \cite{Zhang}, we give the following
\begin{definition}\label{maindef}
The Hall algebra $\mathcal {D}\mathcal {H}_m^{\rm e}(\A)$, called the {\em $m$-periodic extended derived Hall algebra} of $\A$, is the $\mathbb{C}$-vector space with the basis $\{u_{\bigoplus\limits_{i\in\mathbb{Z}_m}M_i[i]}\prod\limits_{i\in\mathbb{Z}_m}K_{\alpha_i,i}~|~[M_i]\in {\rm Iso}(\A),\alpha_i\in K(\A)~\text{for~all~}i\in\mathbb{Z}_m\}$, and with the multiplication defined on basis elements by
{\begin{equation}\label{exthalmul}
\begin{split}&(u_{\bigoplus\limits_{i\in\mathbb{Z}_m}A_i[i]}\prod\limits_{i\in\mathbb{Z}_m}K_{\alpha_i,i})(u_{\bigoplus\limits_{i\in\mathbb{Z}_m}B_i[i]}\prod\limits_{i\in\mathbb{Z}_m}K_{\beta_i,i})=\\&
v^{a_0}\sum\limits_{[I_i],[M_i]\in{\rm Iso}(\A), i\in\mathbb{Z}_m}v^{-(\hat{I}_{m-1},\alpha_0+\beta_0)+\sum\limits_{i=1}^{m-1}(\hat{I}_i,\alpha_{i-1}+\beta_{i-1})+\sum\limits_{i\in\mathbb{Z}_m}\lr{\hat{M}_i-\hat{M}_{i+1},\hat{I}_i}+\sum\limits_{i=1}^{m-1}\lr{\hat{I}_{i-1},\hat{I}_i}-\lr{\hat{I}_0,\hat{I}_{m-1}}}
\\&\quad\quad\quad\quad\quad\quad\quad\quad\quad\quad\quad\prod\limits_{i\in\mathbb{Z}_m}\frac{H_{I_i[1]\oplus A_i,B_i\oplus I_{i-1}[-1]}^{M_i}}{a_{I_i}} u_{\bigoplus\limits_{i\in\mathbb{Z}_m}M_i[i]}\prod\limits_{i\in\mathbb{Z}_m}K_{{\hat{I}}_i+\alpha_i+\beta_i,i},\end{split}\end{equation}}
where $a_0=\sum\limits_{i\in\mathbb{Z}_m}\lr{\hat{A}_i,\hat{B}_i}+\sum\limits_{i\in\mathbb{Z}_m}(\alpha_i,\hat{B}_i-\hat{B}_{i+1})+\sum\limits_{i=1}^{m-1}(\alpha_i,\beta_{i-1})-(\alpha_{m-1},\beta_0)$.
By convention, $\sum\limits_{i=1}^{m-1}x_i=x_1$, if $m=1$.
In particular, we have that
{\begin{equation}\label{yschengfa}
\begin{split}&u_{\bigoplus\limits_{i\in\mathbb{Z}_m}A_i[i]}u_{\bigoplus\limits_{i\in\mathbb{Z}_m}B_i[i]}=
v^{\sum\limits_{i\in\mathbb{Z}_m}\lr{\hat{A}_i,\hat{B}_i}}\sum\limits_{[I_i],[M_i]\in{\rm Iso}(\A), i\in\mathbb{Z}_m}v^{\sum\limits_{i\in\mathbb{Z}_m}\lr{\hat{M}_i-\hat{M}_{i+1},\hat{I}_i}+\sum\limits_{i=1}^{m-1}\lr{\hat{I}_{i-1},\hat{I}_i}-\lr{\hat{I}_0,\hat{I}_{m-1}}}\\&
\quad\quad\quad\quad\quad\quad\quad\quad\quad\quad\quad\quad\quad\quad\prod\limits_{i\in\mathbb{Z}_m}\frac{H_{I_i[1]\oplus A_i,B_i\oplus I_{i-1}[-1]}^{M_i}}{a_{I_i}} u_{\bigoplus\limits_{i\in\mathbb{Z}_m}M_i[i]}\prod\limits_{i\in\mathbb{Z}_m}K_{{\hat{I}}_i,i},\end{split}\end{equation}}
{\begin{equation}\label{Kjiaohuan}
\begin{split}
&\prod\limits_{i\in\mathbb{Z}_m}K_{\alpha_i,i}\prod\limits_{i\in\mathbb{Z}_m}K_{\beta_i,i}=
v^{-(\alpha_{m-1},\beta_0)+\sum\limits_{i=1}^{m-1}(\alpha_i,\beta_{i-1})}\prod\limits_{i\in\mathbb{Z}_m}K_{\alpha_i+\beta_i,i},\end{split}\end{equation}}
{\begin{equation}\label{kujiaohuan}
\begin{split}&(\prod\limits_{i\in\mathbb{Z}_m}K_{\alpha_i,i})u_{\bigoplus\limits_{i\in\mathbb{Z}_m}B_i[i]}=
v^{\sum\limits_{i\in\mathbb{Z}_m}(\alpha_i,\hat{B}_i-\hat{B}_{i+1})}u_{\bigoplus\limits_{i\in\mathbb{Z}_m}B_i[i]}\prod\limits_{i\in\mathbb{Z}_m}K_{\alpha_i,i}.
\end{split}\end{equation}}
\end{definition}
\begin{remark}\label{12zhuji}
In Definition \ref{maindef}, by (\ref{Kjiaohuan}), it is easy to see that
\begin{equation}\label{kkjiaohuan}
\prod\limits_{i\in\mathbb{Z}_m}K_{\alpha_i,i}\prod\limits_{i\in\mathbb{Z}_m}K_{\beta_i,i}
=v^{\sum\limits_{i\in\mathbb{Z}_m}(\alpha_i,\beta_{i-1}-\beta_{i+1})}\prod\limits_{i\in\mathbb{Z}_m}K_{\beta_i,i}\prod\limits_{i\in\mathbb{Z}_m}K_{\alpha_i,i}.
\end{equation}
Hence, if $m=1,2$, the $K$-elements exchange with each other in the Hall algebra $\mathcal {D}\mathcal {H}_m^{\rm e}(\A)$.

For $m=1$, the equation (\ref{kujiaohuan}) becomes $K_{\alpha_0}u_{B_0}=u_{B_0}K_{\alpha_0}$, and thus the $K$-elements in $\mathcal {D}\mathcal {H}_1^{\rm e}(\A)$ are central.
The equation (\ref{yschengfa}) becomes
\begin{equation}\label{bushitensor}
u_{A_0}u_{B_0}=v^{\lr{\hat{A}_0,\hat{B}_0}}\sum\limits_{[I_0],[M_0]\in{\rm Iso}(\A)}\frac{H_{I_0[1]\oplus A_0,B_0\oplus I_0[-1]}^{M_0}}{a_{I_0}}u_{M_0}K_{\hat{I}_0,0}.\end{equation}
Although the $K$-elements are central in the $1$-periodic case, we remark that the $1$-periodic extended derived Hall algebra $\mathcal {D}\mathcal {H}_1^{\rm e}(\A)$ is not isomorphic to the tensor algebra of the $1$-periodic derived Hall algebra $\mathcal{DH}_1(\A)$ (cf. Definition \ref{jidef}) and the group algebra $\mathbb{C}[K(\A)]$ of the Grothendieck group $K(\A)$, since there are $K$-elements appearing in (\ref{bushitensor}).

For $m=2$, it is easy to see that $\hat{M}_0-\hat{M}_1=\hat{A}_0-\hat{A}_1+\hat{B}_0-\hat{B}_1$ in (\ref{exthalmul}), and the Hall algebra $\mathcal {D}\mathcal {H}_2^{\rm e}(\A)$ is the same as given in \cite[Definition 3.2]{Zhang}.
\end{remark}

\begin{theorem}\label{mainresult0}
The $m$-periodic extended derived Hall algebra $\mathcal {D}\mathcal {H}_m^{\rm e}(\A)$ is an associative algebra.
\end{theorem}
\begin{proof}
For $m=1$, by Remark \ref{12zhuji}, the $K$-elements in $\mathcal {D}\mathcal {H}_1^{\rm e}(\A)$ are central, and it is easy to see that it is associative by \cite[Theorem 5.2]{Zhang}. Now we assume that $m>1$.

We need to prove that
\begin{equation}\label{jhlyz}
\begin{split}
&[(u_{\bigoplus\limits_{i\in\mathbb{Z}_m}A_i[i]}\prod\limits_{i\in\mathbb{Z}_m}K_{\alpha_i,i})(u_{\bigoplus\limits_{i\in\mathbb{Z}_m}B_i[i]}\prod\limits_{i\in\mathbb{Z}_m}K_{\beta_i,i})](u_{\bigoplus\limits_{i\in\mathbb{Z}_m}C_i[i]}\prod\limits_{i\in\mathbb{Z}_m}K_{\gamma_i,i})\\
&=(u_{\bigoplus\limits_{i\in\mathbb{Z}_m}A_i[i]}\prod\limits_{i\in\mathbb{Z}_m}K_{\alpha_i,i})[(u_{\bigoplus\limits_{i\in\mathbb{Z}_m}B_i[i]}\prod\limits_{i\in\mathbb{Z}_m}K_{\beta_i,i})(u_{\bigoplus\limits_{i\in\mathbb{Z}_m}C_i[i]}\prod\limits_{i\in\mathbb{Z}_m}K_{\gamma_i,i})]
\end{split}\end{equation}
for any $A_i,B_i,C_i\in{\rm Iso}(\A),\alpha_i,\beta_i,\gamma_i\in K(\A), i\in\mathbb{Z}_m$.

By definition, the left hand side of (\ref{jhlyz}) is
\begin{flalign*}
&\mbox{LHS~~of}~~(\ref{jhlyz})=\\&v^{\sum\limits_{i\in \mathbb{Z}_m}(\alpha_i,\hat{B}_i-\hat{B}_{i+1})-(\alpha_{m-1},\beta_0)+\sum\limits_{i=1}^{m-1}(\alpha_i,\beta_{i-1})+\sum\limits_{i\in\mathbb{Z}_m}(\alpha_i+\beta_i,\hat{C}_i-\hat{C}_{i+1})-(\alpha_{m-1}+\beta_{m-1},\gamma_0)+\sum\limits_{i=1}^{m-1}(\alpha_i+\beta_i,\gamma_{i-1})}\\&
\quad\quad\quad\quad\quad[(u_{\bigoplus\limits_{i\in\mathbb{Z}_m}A_i[i]}u_{\bigoplus\limits_{i\in\mathbb{Z}_m}B_i[i]})u_{\bigoplus\limits_{i\in\mathbb{Z}_m}C_i[i]}]\prod\limits_{i\in\mathbb{Z}_m}K_{\alpha_i+\beta_i+\gamma_i,i}
\end{flalign*}
and
the right hand side of (\ref{jhlyz}) is
\begin{flalign*}
&\mbox{RHS~~of}~~(\ref{jhlyz})=\\&v^{\sum\limits_{i\in \mathbb{Z}_m}(\beta_i,\hat{C}_i-\hat{C}_{i+1})-(\beta_{m-1},\gamma_0)+\sum\limits_{i=1}^{m-1}(\beta_i,\gamma_{i-1})+\sum\limits_{i\in\mathbb{Z}_m}(\alpha_i,\hat{B}_i-\hat{B}_{i+1}+\hat{C}_i-\hat{C}_{i+1})-(\alpha_{m-1},\beta_0+\gamma_0)+\sum\limits_{i=1}^{m-1}(\alpha_i,\beta_{i-1}+\gamma_{i-1})}\\&
\quad\quad\quad\quad\quad[u_{\bigoplus\limits_{i\in\mathbb{Z}_m}A_i[i]}(u_{\bigoplus\limits_{i\in\mathbb{Z}_m}B_i[i]}u_{\bigoplus\limits_{i\in\mathbb{Z}_m}C_i[i]})]\prod\limits_{i\in\mathbb{Z}_m}K_{\alpha_i+\beta_i+\gamma_i,i}.
\end{flalign*}
Hence, the proof of (\ref{jhlyz}) is reduced to prove that
\begin{equation}\label{rjhl}
[(u_{\bigoplus\limits_{i\in\mathbb{Z}_m}A_i[i]}u_{\bigoplus\limits_{i\in\mathbb{Z}_m}B_i[i]})u_{\bigoplus\limits_{i\in\mathbb{Z}_m}C_i[i]}]
=[u_{\bigoplus\limits_{i\in\mathbb{Z}_m}A_i[i]}(u_{\bigoplus\limits_{i\in\mathbb{Z}_m}B_i[i]}u_{\bigoplus\limits_{i\in\mathbb{Z}_m}C_i[i]})].\end{equation}

On the one hand,
\begin{flalign*}
\mbox{LHS~~of}~~(\ref{rjhl})=
\sum\limits_{[I_i],[X_i]\in{\rm Iso}(\A),i\in\mathbb{Z}_m}v^{a_0}\zprod\frac{H_{I_i[1]\oplus A_i,B_i\oplus I_{i-1}[-1]}^{X_i}}{a_{I_i}}(u_{\osum X_i[i]}\zprod K_{\hat{I}_i,i})u_{\bigoplus\limits_{i\in\mathbb{Z}_m}C_i[i]},
\end{flalign*}
where $a_0=\zsum\lr{\hat{A}_i,\hat{B}_i}+\zsum\lr{\hat{X}_i-\hat{X}_{i+1},\hat{I}_i}+\sum\limits_{i=1}^{m-1}\lr{\hat{I}_{i-1},\hat{I}_i}-\lr{\hat{I}_0,\hat{I}_{m-1}}.$ Thus,
\begin{flalign*}
\mbox{LHS~~of}~~(\ref{rjhl})=
\sum\limits_{[I_i],[X_i]\in{\rm Iso}(\A),i\in\mathbb{Z}_m}v^{a}\zprod\frac{H_{I_i[1]\oplus A_i,B_i\oplus I_{i-1}[-1]}^{X_i}}{a_{I_i}}(u_{\osum X_i[i]}u_{\bigoplus\limits_{i\in\mathbb{Z}_m}C_i[i]})\zprod K_{\hat{I}_i,i},
\end{flalign*}
where $a=a_0+\zsum(\hat{I}_i,\hat{C}_i-\hat{C}_{i+1}).$
Hence,
\begin{flalign*}
&\mbox{LHS~~of}~~(\ref{rjhl})=\\
&\sum\limits_{[I_i],[X_i],[J_i],[M_i]\in{\rm Iso}(\A),i\in\mathbb{Z}_m}v^{b_0}\zprod\frac{H_{I_i[1]\oplus A_i,B_i\oplus I_{i-1}[-1]}^{X_i}}{a_{I_i}}
\zprod\frac{H_{J_i[1]\oplus X_i,C_i\oplus J_{i-1}[-1]}^{M_i}}{a_{J_i}}\\
&\quad\quad\quad\quad\quad\quad\quad\quad\quad\quad\quad u_{\osum M_i[i]}\zprod K_{\hat{J}_i,i}\zprod K_{\hat{I}_i,i},
\end{flalign*}
where $b_0=a+\zsum\lr{\hat{X}_i,\hat{C}_i}+\zsum\lr{\hat{M}_i-\hat{M}_{i+1},\hat{J}_i}+\sum\limits_{i=1}^{m-1}\lr{\hat{J}_{i-1},\hat{J}_i}-\lr{\hat{J}_0,\hat{J}_{m-1}}.$
So,
\begin{flalign*}
&\mbox{LHS~~of}~~(\ref{rjhl})=\\
&\sum\limits_{[I_i],[X_i],[J_i],[M_i]\in{\rm Iso}(\A),i\in\mathbb{Z}_m}v^{b}\zprod\frac{H_{I_i[1]\oplus A_i,B_i\oplus I_{i-1}[-1]}^{X_i}}{a_{I_i}}
\zprod\frac{H_{J_i[1]\oplus X_i,C_i\oplus J_{i-1}[-1]}^{M_i}}{a_{J_i}}\\
&\quad\quad\quad\quad\quad\quad\quad\quad\quad\quad\quad u_{\osum M_i[i]}\zprod K_{\hat{I}_i+\hat{J}_i,i},
\end{flalign*}
where $b=b_0-(\hat{J}_{m-1},\hat{I}_0)+\sum\limits_{i=1}^{m-1}(\hat{J}_i,\hat{I}_{i-1}).$

Using Lemma \ref{hallzhuan}(1), we obtain that
\begin{flalign*}
&\mbox{LHS~~of}~~(\ref{rjhl})=
\sum\limits_{[I_i],[X_i],[J_i],[M_i]\in{\rm Iso}(\A),i\in\mathbb{Z}_m}v^{b}q^c\zprod\frac{a_{A_i}a_{B_i}a_{I_{i-1}}}{a_{X_i}}\zprod F_{A_i,I_{i-1}[-1],I_i[1],B_i}^{X_i}\\
&\quad\quad\quad\quad\quad\quad\quad\quad\quad \zprod\frac{a_{X_i}a_{C_i}a_{J_{i-1}}}{a_{M_i}}\zprod F_{X_i,J_{i-1}[-1],J_i[1],C_i}^{M_i}
u_{\osum M_i[i]}\zprod K_{\hat{I}_i+\hat{J}_i,i},
\end{flalign*}
where $c=\zsum (\lr{\hat{I}_{i-1},\hat{I}_i}-\lr{\hat{A}_i,\hat{I}_i}-\lr{\hat{I}_{i-1},\hat{B}_i}+\lr{\hat{J}_{i-1},\hat{J}_i}-\lr{\hat{X}_i,\hat{J}_i}-\lr{\hat{J}_{i-1},\hat{C}_i})$.
Furthermore, using Lemma \ref{hallzhuan}(2), we get that
\begin{flalign*}
&\mbox{LHS~~of}~~(\ref{rjhl})=\zprod (a_{A_i}a_{B_i}a_{C_i})
\sum\limits_{[I_i],[X_i],[J_i],[M_i]\in{\rm Iso}(\A),i\in\mathbb{Z}_m}v^{b}q^c\zprod\frac{a_{I_{i-1}}a_{J_{i-1}}}{a_{M_i}}\\
&\zprod(\sum\limits_{[N_i],[L_i]\in{\rm Iso}(\A)}\frac{a_{N_i}a_{L_i}}{a_{A_i}a_{B_i}}F_{I_{i-1},N_{i}}^{A_i}F_{L_i,I_i}^{B_i}F_{N_i,L_i}^{X_i})
\zprod(\sum\limits_{[S_i],[T_i]\in{\rm Iso}(\A)}\frac{a_{S_i}a_{T_i}}{a_{X_i}a_{C_i}}F_{J_{i-1},S_{i}}^{X_i}F_{T_i,J_i}^{C_i}F_{S_i,T_i}^{M_i})\\
&\quad\quad\quad\quad\quad\quad\quad\quad\quad\quad\quad\quad\quad\quad\quad\quad\quad\quad\quad\quad u_{\osum M_i[i]}\zprod K_{\hat{I}_i+\hat{J}_i,i}.
\end{flalign*}
Thus,
\begin{flalign*}
&\mbox{LHS~~of}~~(\ref{rjhl})=
\sum\limits_{[I_i],[X_i],[J_i],[M_i],[N_i],[L_i],[S_i],[T_i]\in{\rm Iso}(\A),i\in\mathbb{Z}_m}v^{b}q^c\zprod\frac{a_{I_{i-1}}a_{J_{i-1}}a_{N_i}a_{L_i}a_{S_i}a_{T_i}}{a_{M_i}}\\
&\quad\quad\zprod(F_{I_{i-1},N_{i}}^{A_i}F_{L_i,I_i}^{B_i}F_{T_i,J_i}^{C_i}F_{S_i,T_i}^{M_i})
\zprod(F_{N_{i},L_{i}}^{X_i}F_{J_{i-1},S_i}^{X_i}\frac{1}{a_{X_i}}) u_{\osum M_i[i]}\zprod K_{\hat{I}_i+\hat{J}_i,i}.
\end{flalign*}
For each $i\in\mathbb{Z}_m$, using Green's formula, we have that
\begin{flalign*}
&\sum\limits_{[X_i]\in{\rm Iso}(\A)}F_{N_{i},L_{i}}^{X_i}F_{J_{i-1},S_i}^{X_i}\frac{1}{a_{X_i}}\\&=\sum\limits_{[U_i],[V_i],[W_i],[Z_i]\in{\rm Iso}(\A)}
q^{-\lr{\hat{U}_i,\hat{Z}_i}}F_{U_i,V_i}^{N_i}F_{W_i,Z_i}^{L_i}F_{U_i,W_i}^{J_{i-1}}F_{V_i,Z_i}^{S_i}\frac{a_{U_i}a_{V_i}a_{W_i}a_{Z_i}}{a_{N_i}a_{L_i}a_{J_{i-1}}a_{S_i}}.
\end{flalign*}
Noting that $\hat{X}_i=\hat{A}_i+\hat{B}_i-\hat{I}_i-\hat{I}_{i-1}$ for each $i\in\mathbb{Z}_m$, we obtain that
\begin{flalign*}
&\mbox{LHS~~of}~~(\ref{rjhl})=
\sum\limits_{[I_i],[J_i],[M_i],[N_i],[L_i],[S_i],[T_i],[U_i],[V_i],[W_i],[Z_i]\in{\rm Iso}(\A),i\in\mathbb{Z}_m}v^{b}q^{c-\sum\limits_{i\in\mathbb{Z}_m}\lr{\hat{U}_i,\hat{Z}_i}}\\&\zprod\frac{a_{I_{i-1}}a_{T_i}a_{U_i}a_{V_i}a_{W_i}a_{Z_i}}{a_{M_i}}\zprod(F_{I_{i-1},N_{i}}^{A_i}F_{L_i,I_i}^{B_i}F_{T_i,J_i}^{C_i}F_{S_i,T_i}^{M_i}
F_{U_i,V_i}^{N_i}F_{W_i,Z_i}^{L_i}F_{U_i,W_i}^{J_{i-1}}F_{V_i,Z_i}^{S_i})\\&
\quad\quad\quad\quad\quad\quad\quad\quad\quad\quad\quad\quad\quad\quad u_{\osum M_i[i]}\zprod K_{\hat{I}_i+\hat{J}_i,i}.
\end{flalign*}
Noting that $\hat{J}_i=\hat{U}_{i+1}+\hat{W}_{i+1}$ for each $i\in\mathbb{Z}_m$ and applying the associativity formula (\ref{associativity}), we get that
\begin{equation}\label{zuobian}
\begin{split}
&\mbox{LHS~~of}~~(\ref{rjhl})=
\sum\limits_{[I_i],[M_i],[T_i],[U_i],[V_i],[W_i],[Z_i]\in{\rm Iso}(\A),i\in\mathbb{Z}_m}v^{b}q^{c-\sum\limits_{i\in\mathbb{Z}_m}\lr{\hat{U}_i,\hat{Z}_i}}\zprod\frac{a_{I_{i-1}}a_{T_i}a_{U_i}a_{V_i}a_{W_i}a_{Z_i}}{a_{M_i}}\\
&\quad\quad\quad\zprod(F_{I_{i-1},U_{i},V_i}^{A_i}F_{W_i,Z_i,I_i}^{B_i}F_{T_i,U_{i+1},W_{i+1}}^{C_i}F_{V_i,Z_i,T_i}^{M_i}) u_{\osum M_i[i]}\zprod K_{\hat{I}_i+\hat{U}_{i+1}+\hat{W}_{i+1},i}.
\end{split}
\end{equation}

On the other hand,
\begin{flalign*}
&\mbox{RHS~~of}~~(\ref{rjhl})=\\
&\sum\limits_{[J_i],[Y_i]\in{\rm Iso}(\A),i\in\mathbb{Z}_m}v^{x_0}\zprod\frac{H_{J_i[1]\oplus B_i,C_i\oplus J_{i-1}[-1]}^{Y_i}}{a_{J_i}}u_{\bigoplus\limits_{i\in\mathbb{Z}_m}A_i[i]}(u_{\osum Y_i[i]}\zprod K_{\hat{J}_i,i}),
\end{flalign*}
where $x_0=\zsum\lr{\hat{B}_i,\hat{C}_i}+\zsum\lr{\hat{Y}_i-\hat{Y}_{i+1},\hat{J}_i}+\sum\limits_{i=1}^{m-1}\lr{\hat{J}_{i-1},\hat{J}_i}-\lr{\hat{J}_0,\hat{J}_{m-1}}.$ Thus,
\begin{flalign*}
&\mbox{RHS~~of}~~(\ref{rjhl})=\\
&\sum\limits_{[J_i],[Y_i],[I_i],[M_i]\in{\rm Iso}(\A),i\in\mathbb{Z}_m}v^{x}\zprod\frac{H_{I_i[1]\oplus A_i,Y_i\oplus I_{i-1}[-1]}^{M_i}}{a_{I_i}}
\zprod\frac{H_{J_i[1]\oplus B_i,C_i\oplus J_{i-1}[-1]}^{Y_i}}{a_{J_i}}\\
&\quad\quad\quad\quad\quad\quad\quad\quad\quad\quad\quad u_{\osum M_i[i]}\zprod K_{\hat{I}_i,i}\zprod K_{\hat{J}_i,i},
\end{flalign*}
where $x=x_0+\zsum\lr{\hat{A}_i,\hat{Y}_i}+\zsum\lr{\hat{M}_i-\hat{M}_{i+1},\hat{I}_i}+\sum\limits_{i=1}^{m-1}\lr{\hat{I}_{i-1},\hat{I}_i}-\lr{\hat{I}_0,\hat{I}_{m-1}}.$
So,
\begin{flalign*}
&\mbox{RHS~~of}~~(\ref{rjhl})=\\
&\sum\limits_{[I_i],[Y_i],[J_i],[M_i]\in{\rm Iso}(\A),i\in\mathbb{Z}_m}v^{y}\zprod\frac{H_{I_i[1]\oplus A_i,Y_i\oplus I_{i-1}[-1]}^{M_i}}{a_{I_i}}
\zprod\frac{H_{J_i[1]\oplus B_i,C_i\oplus J_{i-1}[-1]}^{Y_i}}{a_{J_i}}\\
&\quad\quad\quad\quad\quad\quad\quad\quad\quad\quad\quad u_{\osum M_i[i]}\zprod K_{\hat{I}_i+\hat{J}_i,i},
\end{flalign*}
where $y=x-(\hat{I}_{m-1},\hat{J}_0)+\sum\limits_{i=1}^{m-1}(\hat{I}_i,\hat{J}_{i-1}).$

Using Lemma \ref{hallzhuan}(1), we obtain that
\begin{flalign*}
&\mbox{RHS~~of}~~(\ref{rjhl})=
\sum\limits_{[I_i],[Y_i],[J_i],[M_i]\in{\rm Iso}(\A),i\in\mathbb{Z}_m}v^{y}q^z\zprod\frac{a_{A_i}a_{Y_i}a_{I_{i-1}}}{a_{M_i}}\zprod F_{A_i,I_{i-1}[-1],I_i[1],Y_i}^{M_i}\\
&\quad\quad\quad\quad\quad\quad\quad\quad\quad \zprod\frac{a_{B_i}a_{C_i}a_{J_{i-1}}}{a_{Y_i}}\zprod F_{B_i,J_{i-1}[-1],J_i[1],C_i}^{Y_i}
u_{\osum M_i[i]}\zprod K_{\hat{I}_i+\hat{J}_i,i},
\end{flalign*}
where $z=\zsum (\lr{\hat{I}_{i-1},\hat{I}_i}-\lr{\hat{A}_i,\hat{I}_i}-\lr{\hat{I}_{i-1},\hat{Y}_i}+\lr{\hat{J}_{i-1},\hat{J}_i}-\lr{\hat{B}_i,\hat{J}_i}-\lr{\hat{J}_{i-1},\hat{C}_i})$. Furthermore, using Lemma \ref{hallzhuan}(2), we get that
\begin{flalign*}
&\mbox{RHS~~of}~~(\ref{rjhl})=\zprod (a_{A_i}a_{B_i}a_{C_i})
\sum\limits_{[I_i],[Y_i],[J_i],[M_i]\in{\rm Iso}(\A),i\in\mathbb{Z}_m}v^{y}q^z\zprod\frac{a_{I_{i-1}}a_{J_{i-1}}}{a_{M_i}}\\
&\zprod(\sum\limits_{[N_i],[L_i]\in{\rm Iso}(\A)}\frac{a_{N_i}a_{L_i}}{a_{A_i}a_{Y_i}}F_{I_{i-1},N_{i}}^{A_i}F_{L_i,I_i}^{Y_i}F_{N_i,L_i}^{M_i})
\zprod(\sum\limits_{[S_i],[T_i]\in{\rm Iso}(\A)}\frac{a_{S_i}a_{T_i}}{a_{B_i}a_{C_i}}F_{J_{i-1},S_{i}}^{B_i}F_{T_i,J_i}^{C_i}F_{S_i,T_i}^{Y_i})\\
&\quad\quad\quad\quad\quad\quad\quad\quad\quad\quad\quad\quad\quad\quad\quad\quad\quad\quad\quad\quad u_{\osum M_i[i]}\zprod K_{\hat{I}_i+\hat{J}_i,i}.
\end{flalign*}
Thus,
\begin{flalign*}
&\mbox{RHS~~of}~~(\ref{rjhl})=
\sum\limits_{[I_i],[Y_i],[J_i],[M_i],[N_i],[L_i],[S_i],[T_i]\in{\rm Iso}(\A),i\in\mathbb{Z}_m}v^{y}q^z\zprod\frac{a_{I_{i-1}}a_{J_{i-1}}a_{N_i}a_{L_i}a_{S_i}a_{T_i}}{a_{M_i}}\\
&\zprod(F_{I_{i-1},N_{i}}^{A_i}F_{J_{i-1},S_i}^{B_i}F_{T_i,J_i}^{C_i}F_{N_i,L_i}^{M_i})
\zprod(F_{L_{i},I_{i}}^{Y_i}F_{S_{i},T_i}^{Y_i}\frac{1}{a_{Y_i}}) u_{\osum M_i[i]}\zprod K_{\hat{I}_i+\hat{J}_i,i}.
\end{flalign*}
For each $i\in\mathbb{Z}_m$, using Green's formula, we have that
\begin{flalign*}
&\sum\limits_{[Y_i]\in{\rm Iso}(\A)}F_{L_{i},I_{i}}^{Y_i}F_{S_{i},T_i}^{Y_i}\frac{1}{a_{Y_i}}\\&=\sum\limits_{[K_i],[P_i],[Q_i],[R_i]\in{\rm Iso}(\A)}
q^{-\lr{\hat{K}_i,\hat{R}_i}}F_{K_i,P_i}^{L_i}F_{Q_i,R_i}^{I_i}F_{K_i,Q_i}^{S_{i}}F_{P_i,R_i}^{T_i}\frac{a_{K_i}a_{P_i}a_{Q_i}a_{R_i}}{a_{L_i}a_{I_i}a_{S_{i}}a_{T_i}}.
\end{flalign*}
Noting that $\hat{Y}_i=\hat{B}_i+\hat{C}_i-\hat{J}_i-\hat{J}_{i-1}$ for each $i\in\mathbb{Z}_m$, we obtain that
\begin{flalign*}
&\mbox{RHS~~of}~~(\ref{rjhl})=
\sum\limits_{[I_i],[J_i],[M_i],[N_i],[L_i],[S_i],[T_i],[K_i],[P_i],[Q_i],[R_i]\in{\rm Iso}(\A),i\in\mathbb{Z}_m}v^{y}q^{z-\sum\limits_{i\in\mathbb{Z}_m}\lr{\hat{K}_i,\hat{R}_i}}\\&\zprod\frac{a_{J_{i-1}}a_{N_i}a_{K_i}a_{P_i}a_{Q_i}a_{R_i}}{a_{M_i}}\zprod(F_{I_{i-1},N_{i}}^{A_i}F_{J_{i-1},S_i}^{B_i}F_{T_i,J_i}^{C_i}F_{N_i,L_i}^{M_i}
F_{K_i,P_i}^{L_i}F_{Q_i,R_i}^{I_i}F_{K_i,Q_i}^{S_{i}}F_{P_i,R_i}^{T_i})\\&
\quad\quad\quad\quad\quad\quad\quad\quad\quad\quad\quad\quad\quad\quad u_{\osum M_i[i]}\zprod K_{\hat{I}_i+\hat{J}_i,i}.
\end{flalign*}
Noting that $\hat{I}_i=\hat{Q}_{i}+\hat{R}_{i}$ for each $i\in\mathbb{Z}_m$ and applying the associativity formula (\ref{associativity}), we get that
\begin{equation}\label{youbian}
\begin{split}
&\mbox{RHS~~of}~~(\ref{rjhl})=
\sum\limits_{[J_i],[M_i],[N_i],[K_i],[P_i],[Q_i],[R_i]\in{\rm Iso}(\A),i\in\mathbb{Z}_m}v^{y}q^{z-\sum\limits_{i\in\mathbb{Z}_m}\lr{\hat{K}_i,\hat{R}_i}}\zprod\frac{a_{J_{i-1}}a_{N_i}a_{K_i}a_{P_i}a_{Q_i}a_{R_i}}{a_{M_i}}\\
&\quad\quad\quad\quad\zprod(F_{Q_{i-1},R_{i-1},N_i}^{A_i}F_{J_{i-1},K_i,Q_i}^{B_i}F_{P_i,R_{i},J_{i}}^{C_i}F_{N_i,K_i,P_i}^{M_i}) u_{\osum M_i[i]}\zprod K_{\hat{Q}_i+\hat{R}_{i}+\hat{J}_{i},i}.
\end{split}
\end{equation}

Replacing the notations $I_i, T_i, U_i, V_i, W_i, Z_i$ in (\ref{zuobian}) by $Q_i, P_i, R_{i-1}, N_i, J_{i-1}, K_i$, respectively, for each $i\in\mathbb{Z}_m$,
we conclude that all terms in (\ref{zuobian}) are the same as those in (\ref{youbian}) except the exponents of $v$ and $q$.

Now let us compare the exponents of $v$ and $q$ in (\ref{zuobian}) and (\ref{youbian}).
Noting that $\hat{J}_i=\hat{U}_{i+1}+\hat{W}_{i+1}$ for each $i\in\mathbb{Z}_m$ in the exponent $c$ and replacing the notations, we have that
\begin{equation}\label{lcommon1}
\begin{split}
c=\sum\limits_{i\in\mathbb{Z}_m}&(\lr{\hat{Q}_{i-1},\hat{Q}_i}-\lr{\hat{A}_i,\hat{Q}_i}-\lr{\hat{Q}_{i-1},\hat{B}_i}+\lr{\hat{R}_{i-1}+\hat{J}_{i-1},\hat{R}_i+\hat{J}_i}\\&-\lr{\hat{A}_i+\hat{B}_i,\hat{R}_i+\hat{J}_i}+\lr{\hat{Q}_i+\hat{Q}_{i-1},\hat{R}_i+\hat{J}_i}-\lr{\hat{R}_{i-1}+\hat{J}_{i-1},C_i}).
\end{split}\end{equation}
Noting that $\hat{I}_i=\hat{Q}_{i}+\hat{R}_{i}$ for each $i\in\mathbb{Z}_m$ in the exponent $z$, we have that
\begin{equation}\label{rcommon1}
\begin{split}
z=\sum\limits_{i\in\mathbb{Z}_m}&(\lr{\hat{Q}_{i-1}+\hat{R}_{i-1},\hat{Q}_i+\hat{R}_i}-\lr{\hat{A}_i,\hat{Q}_i+\hat{R}_i}-\lr{\hat{Q}_{i-1}+\hat{R}_{i-1},\hat{B}_i+\hat{C}_i}\\&+\lr{\hat{Q}_{i-1}+\hat{R}_{i-1},\hat{J}_i+\hat{J}_{i-1}}+\lr{\hat{J}_{i-1},\hat{J}_i}-\lr{\hat{B}_i,\hat{J}_i}-\lr{\hat{J}_{i-1},\hat{C}_i}).
\end{split}\end{equation}
By direct calculations, we get that \begin{equation*}c=\sum\limits_{i\in\mathbb{Z}_m}(\lr{\hat{J}_{i-1},\hat{R}_i}-\lr{\hat{A}_i,\hat{J}_i}-\lr{\hat{B}_i,\hat{R}_i}+\lr{\hat{Q}_i,\hat{R}_i})+\text{Common~terms}~(\clubsuit)\end{equation*} and \begin{equation*}z=\sum\limits_{i\in\mathbb{Z}_m}(\lr{\hat{R}_{i-1},\hat{Q}_i}-\lr{\hat{R}_{i-1},\hat{B}_i}-\lr{\hat{Q}_{i-1},\hat{C}_i}+\lr{\hat{R}_{i-1},\hat{J}_{i-1}})+\text{Common~terms}~(\clubsuit),\end{equation*}
where $(\clubsuit)$ denotes the common terms of $(\ref{lcommon1})$ and $(\ref{rcommon1})$.

Using $\hat{X}_i-\hat{X}_{i+1}=\hat{A}_i-\hat{A}_{i+1}+\hat{B}_i-\hat{B}_{i+1}+\hat{I}_{i+1}-\hat{I}_{i-1}$, $\hat{X}_i=\hat{A}_i+\hat{B}_i-\hat{I}_i-\hat{I}_{i-1}$, $\hat{M}_i-\hat{M}_{i+1}=\hat{A}_i-\hat{A}_{i+1}+\hat{B}_i-\hat{B}_{i+1}+\hat{C}_i-\hat{C}_{i+1}+\hat{I}_{i+1}-\hat{I}_{i-1}+\hat{J}_{i+1}-\hat{J}_{i-1}$ and $\hat{J}_i=\hat{U}_{i+1}+\hat{W}_{i+1}$, we can obtain that
\begin{equation}\label{zhishub}
\begin{split}
&b=\sum\limits_{i\in\mathbb{Z}_m}(\lr{\hat{A}_i,\hat{B}_i}+\lr{\hat{A}_i-\hat{A}_{i+1}+\hat{B}_i-\hat{B}_{i+1}+\hat{I}_{i+1}-\hat{I}_{i-1},\hat{I}_i}+(\hat{I}_i,\hat{C}_i-\hat{C}_{i+1})
+\\&\lr{\hat{A}_i-\hat{A}_{i+1}+\hat{B}_i-\hat{B}_{i+1}+\hat{C}_i-\hat{C}_{i+1}+\hat{I}_{i+1}-\hat{I}_{i-1}+\hat{U}_{i+2}+\hat{W}_{i+2}-\hat{U}_i-\hat{W}_i,\hat{U}_{i+1}+\hat{W}_{i+1}}
\\&+\lr{\hat{A}_i+\hat{B}_i-\hat{I}_i-\hat{I}_{i-1},\hat{C}_i})+\sum\limits_{i=1}^{m-1}\lr{\hat{I}_{i-1},\hat{I}_i}-\lr{\hat{I}_0,\hat{I}_{m-1}}+\sum\limits_{i=1}^{m-1}\lr{\hat{U}_i+\hat{W}_i,\hat{U}_{i+1}+\hat{W}_{i+1}}\\
&-\lr{\hat{U}_1+\hat{W}_1,\hat{U}_0+\hat{W}_0}-(\hat{U}_0+\hat{W}_0,\hat{I}_0)+\sum\limits_{i=1}^{m-1}(\hat{U}_{i+1}+\hat{W}_{i+1},\hat{I}_{i-1}).
\end{split}
\end{equation}

Using $\hat{Y}_i-\hat{Y}_{i+1}=\hat{B}_i-\hat{B}_{i+1}+\hat{C}_i-\hat{C}_{i+1}+\hat{J}_{i+1}-\hat{J}_{i-1}$, $\hat{Y}_i=\hat{B}_i+\hat{C}_i-\hat{J}_i-\hat{J}_{i-1}$, $\hat{M}_i-\hat{M}_{i+1}=\hat{A}_i-\hat{A}_{i+1}+\hat{B}_i-\hat{B}_{i+1}+\hat{C}_i-\hat{C}_{i+1}+\hat{I}_{i+1}-\hat{I}_{i-1}+\hat{J}_{i+1}-\hat{J}_{i-1}$ and $\hat{I}_i=\hat{Q}_{i}+\hat{R}_{i}$, we can obtain that
\begin{equation}\label{rcommon2}
\begin{split}
&y=\sum\limits_{i\in\mathbb{Z}_m}(\lr{\hat{B}_i,\hat{C}_i}+\lr{\hat{B}_i-\hat{B}_{i+1}+\hat{C}_i-\hat{C}_{i+1}+\hat{J}_{i+1}-\hat{J}_{i-1},\hat{J}_i}+\lr{\hat{A}_i,\hat{B}_i+\hat{C}_i-\hat{J}_i-\hat{J}_{i-1}}
+\\&\lr{\hat{A}_i-\hat{A}_{i+1}+\hat{B}_i-\hat{B}_{i+1}+\hat{C}_i-\hat{C}_{i+1}+\hat{Q}_{i+1}+\hat{R}_{i+1}-\hat{Q}_{i-1}-\hat{R}_{i-1}+\hat{J}_{i+1}-\hat{J}_{i-1},\hat{Q}_i+\hat{R}_i})
\\&+\sum\limits_{i=1}^{m-1}\lr{\hat{J}_{i-1},\hat{J}_i}-\lr{\hat{J}_0,\hat{J}_{m-1}}+\sum\limits_{i=1}^{m-1}\lr{\hat{Q}_{i-1}+\hat{R}_{i-1},\hat{Q}_{i}+\hat{R}_{i}}
-\lr{\hat{Q}_0+\hat{R}_0,\hat{Q}_{m-1}+\hat{R}_{m-1}}\\
&+\sum\limits_{i=1}^{m-1}(\hat{Q}_{i}+\hat{R}_{i},\hat{J}_{i-1})-(\hat{Q}_{m-1}+\hat{R}_{m-1},\hat{J}_0).
\end{split}
\end{equation}

Replacing the notations $I_i, U_i, W_i$ in (\ref{zhishub}) by $Q_i, R_{i-1}, J_{i-1}$, respectively, for each $i\in\mathbb{Z}_m$, we have that
\begin{equation*}b=\sum\limits_{i\in\mathbb{Z}_m}(\lr{\hat{A}_i,\hat{J}_i}-2\lr{\hat{Q}_{i-1},\hat{C}_i})+\text{Common~terms}~(\spadesuit)\end{equation*} and \begin{equation*}y=-\sum\limits_{i\in\mathbb{Z}_m}\lr{\hat{A}_i,\hat{J}_i}+\text{Common~terms}~(\spadesuit),\end{equation*}
where $(\spadesuit)$ denotes the common terms of $(\ref{zhishub})$ and $(\ref{rcommon2})$.

Hence, the exponent in (\ref{zuobian}) is
$2c+b-2\sum\limits_{i\in\mathbb{Z}_m}\lr{\hat{R}_{i-1},\hat{K}_i}$ after we replace the notations $U_i,Z_i$ by $R_{i-1}, K_i$, respectively, which is equal to
\begin{equation}\label{jieshuzuobian}
\begin{split}
\sum\limits_{i\in\mathbb{Z}_m}&(2\lr{\hat{J}_{i-1},\hat{R}_i}-2\lr{\hat{A}_i,\hat{J}_i}-2\lr{\hat{B}_i,\hat{R}_i}+2\lr{\hat{Q}_i,\hat{R}_i}+\lr{\hat{A}_i,\hat{J}_i}-2\lr{\hat{Q}_{i-1},\hat{C}_i}\\
&-2\lr{\hat{R}_{i-1},\hat{B}_i-\hat{Q}_i-\hat{J}_{i-1}})+2(\clubsuit)+(\spadesuit)
\end{split}
\end{equation}
and the exponent in (\ref{youbian}) is
$2z+y-2\sum\limits_{i\in\mathbb{Z}_m}\lr{\hat{K}_{i},\hat{R}_i}$, which is equal to
\begin{equation}\label{jieshuyoubian}
\begin{split}
\sum\limits_{i\in\mathbb{Z}_m}&(2\lr{\hat{R}_{i-1},\hat{Q}_i}-2\lr{\hat{R}_{i-1},\hat{B}_i}-2\lr{\hat{Q}_{i-1},\hat{C}_i}+2\lr{\hat{R}_{i-1},\hat{J}_{i-1}}-\lr{\hat{A}_i,\hat{J}_i}\\&-2\lr{\hat{B}_i-\hat{Q}_i-\hat{J}_{i-1},\hat{R}_i})
+2(\clubsuit)+(\spadesuit).
\end{split}
\end{equation}
It is easy to see that the formulas (\ref{jieshuzuobian}) and (\ref{jieshuyoubian}) are equal.
Hence, the exponents in (\ref{zuobian}) and (\ref{youbian}) are also the same.
Therefore, we complete the proof.
\end{proof}

For each $i\in\mathbb{Z}_m$, let $\lambda_i:{\mathcal {H}}_{\rm{tw}}^{\rm e}(\mathcal{A})\hookrightarrow\mathcal {D}\mathcal {H}_m^{\rm e}(\A)$ be the injective linear map defined on basis elements by $u_{[M]}K_\alpha\mapsto u_{M[i]}K_{\alpha,i}$.
\begin{lemma}
If $m>1$, then for each $i\in\mathbb{Z}_m$, the map $\lambda_i$ above is an embedding of algebras.
\end{lemma}
\begin{proof}
It is easily proved by the definition of the multiplication in $\mathcal {D}\mathcal {H}_m^{\rm e}(\A)$.
\end{proof}

\begin{proposition}\label{mjiaoji}
The $m$-periodic extended derived Hall algebra $\mathcal {D}\mathcal {H}_m^{\rm e}(\A)$ has a basis given by
$$\{\prod\limits_{i\in\mathbb{Z}_m}u_{A_i[i]}\prod\limits_{i\in\mathbb{Z}_m}K_{\alpha_i,i}~|~[A_i]\in {\rm Iso}(\A),\alpha_i\in K(\A)~\text{for~all~}i\in\mathbb{Z}_m\}.$$
\end{proposition}
\begin{proof}
If $m=1$, it is clear. If $m=2$, it is proved by \cite[Proposition 4.2]{Zhang} and the equation (\ref{kujiaohuan}).
Now, we prove for the case $m>2$ in a similar way as \cite{Zhang}.

For any $[A_i]\in {\rm Iso}(\A),\alpha_i\in K(\A)$ with $i\in\mathbb{Z}_m$, define \begin{flalign*}\delta(u_{\bigoplus\limits_{i\in\mathbb{Z}_m}A_i[i]}\prod\limits_{i\in\mathbb{Z}_m}K_{\alpha_i,i})&=\dim_k\Ext_{D_m(\A)}^1(\widetilde{A}_0\oplus \widetilde{A}_{m-1}[m-1],\widetilde{A}_0\oplus \widetilde{A}_{m-1}[m-1])\\&=\dim_k\Ext_{D_m(\A)}^1(\widetilde{A}_0[1]\oplus \widetilde{A}_{m-1},\widetilde{A}_0[1]\oplus \widetilde{A}_{m-1}).\end{flalign*}

If $\delta(u_{\bigoplus\limits_{i\in\mathbb{Z}_m}A_i[i]}\prod\limits_{i\in\mathbb{Z}_m}K_{\alpha_i,i})=0$, we have that $\Hom_{\A}(A_0,A_{m-1})=0$, since
\begin{flalign*}&\Ext_{D_m(\A)}^1(\widetilde{A}_0\oplus \widetilde{A}_{m-1}[m-1],\widetilde{A}_0\oplus \widetilde{A}_{m-1}[m-1])\\&\cong\Hom_{\A}(A_0, A_{m-1})\oplus\Ext^1_{\A}(A_0,A_0)\oplus\Ext^1_{\A}(A_{m-1},A_{m-1}).\end{flalign*}
Thus, by (\ref{yschengfa}), we obtain that $$\prod\limits_{i\in\mathbb{Z}_m}u_{A_i[i]}\prod\limits_{i\in\mathbb{Z}_m}K_{\alpha_i,i}=u_{\bigoplus\limits_{i=0}^{m-2}A_i[i]}u_{A_{m-1}[m-1]} \prod\limits_{i\in\mathbb{Z}_m}K_{\alpha_i,i}=u_{\osum A_i[i]}\prod\limits_{i\in\mathbb{Z}_m}K_{\alpha_i,i}.$$

If $\delta(u_{\bigoplus\limits_{i\in\mathbb{Z}_m}A_i[i]}\prod\limits_{i\in\mathbb{Z}_m}K_{\alpha_i,i})>0$, by (\ref{yschengfa}), we have that
\begin{equation}\label{basis}
\begin{split}
&\prod\limits_{i\in\mathbb{Z}_m}u_{A_i[i]}\prod\limits_{i\in\mathbb{Z}_m}K_{\alpha_i,i}=u_{\osum A_i[i]}\prod\limits_{i\in\mathbb{Z}_m}K_{\alpha_i,i}+\sum\limits_{[I_{m-1}]\neq[0],[M_0],[M_{m-1}]}v^{\lr{\hat{M}_{m-1}-\hat{M}_0,\hat{I}_{m-1}}}\\
&\quad\quad\frac{H_{A_0,I_{m-1}[-1]}^{M_0}H_{I_{m-1}[1],A_{m-1}}^{M_{m-1}}}{a_{I_{m-1}}}
u_{M_0\bigoplus\bigoplus\limits_{i=1}^{m-2}A_i[i]\bigoplus M_{m-1}[m-1]}K_{\hat{I}_{m-1},m-1}\prod\limits_{i\in\mathbb{Z}_m}K_{\alpha_i,i}.\end{split}
\end{equation}
For any $[I_{m-1}],[M_0],[M_{m-1}]\in\Iso(\A)$, if $H_{A_0,I_{m-1}[-1]}^{M_0}H_{I_{m-1}[1],A_{m-1}}^{M_{m-1}}\neq0$, then we have an exact sequence in $\A$
$$\xymatrix@R=0.8pc{0\ar[r]&M_0\ar[r]&A_0\ar[rr]^-{f}\ar@{->>}[rd]&&A_{m-1}\ar[r]&M_{m-1}\ar[r]&0.\\
&&&I_{m-1}\ar@{>->}[ru]&&&}$$
Thus, we obtain a triangle in $D^b(\A)$
$$\xymatrix{A_0\ar[r]^-f&A_{m-1}\ar[r]&M_0[1]\oplus M_{m-1}\ar[r]&A_0[1],}$$ which gives a triangle in $D_m(\A)$
\begin{equation}\label{kelie}
\xymatrix{\widetilde{A}_0\ar[r]^-{\widetilde{f}}&\widetilde{A}_{m-1}\ar[r]&\widetilde{M}_0[1]\oplus \widetilde{M}_{m-1}\ar[r]&\widetilde{A}_0[1].}\end{equation}
By \cite[Lemma 4.2]{RuanZ}, we have that $$\dim_k\Ext_{D_m(\A)}^1(\widetilde{M}_0[1]\oplus \widetilde{M}_{m-1},\widetilde{M}_0[1]\oplus \widetilde{M}_{m-1})\leq\dim_k\Ext_{D_m(\A)}^1(\widetilde{A}_0[1]\oplus \widetilde{A}_{m-1},\widetilde{A}_0[1]\oplus \widetilde{A}_{m-1})$$
and the equality holds if and only if the triangle (\ref{kelie}) splits if and only if $\widetilde{f}=0$, which is also equivalent to $f=0$, since $\A$ is fully embedded into $D_m(\A)$ for $m>2$.

Hence, for each term in the sum $\sum$ of (\ref{basis}), its degree $\delta$ is less than $\delta(u_{\osum A_i[i]}\prod\limits_{i\in\mathbb{Z}_m}K_{\alpha_i,i})$. Since $\{u_{\bigoplus\limits_{i\in\mathbb{Z}_m}A_i[i]}\prod\limits_{i\in\mathbb{Z}_m}K_{\alpha_i,i}~|~[A_i]\in {\rm Iso}(\A),\alpha_i\in K(\A)~\text{for~all~}i\in\mathbb{Z}_m\}$ is a basis of $\mathcal {D}\mathcal {H}_m^{\rm e}(\A)$, according to (\ref{basis}), we obtain that $$\mathfrak{B}:=\{\prod\limits_{i\in\mathbb{Z}_m}u_{A_i[i]}\prod\limits_{i\in\mathbb{Z}_m}K_{\alpha_i,i}~|~[A_i]\in {\rm Iso}(\A),\alpha_i\in K(\A)~\text{for~all~}i\in\mathbb{Z}_m\}$$ is a linear independent set. By induction on $\delta(u_{\osum A_i[i]}\prod\limits_{i\in\mathbb{Z}_m}K_{\alpha_i,i})$ for all $[A_i]\in {\rm Iso}(\A),\alpha_i\in K(\A)$, we get that $u_{\osum A_i[i]}\prod\limits_{i\in\mathbb{Z}_m}K_{\alpha_i,i}$ can be written as a linear combination of elements in $\mathfrak{B}$. Therefore, $\mathfrak{B}$ is a basis of $\mathcal {D}\mathcal {H}_m^{\rm e}(\A)$.
\end{proof}
\begin{corollary}\label{mcopy}
The multiplication map induces an isomorphism of vector spaces
\begin{equation}\begin{aligned}
&\mu:\bigotimes\limits_{i\in\mathbb{Z}_m}{\mathcal {H}}_{\rm{tw}}^{\rm e}(\mathcal{A})\longrightarrow \mathcal {D}\mathcal {H}_m^{\rm e}(\A),
&\bigotimes\limits_{i\in\mathbb{Z}_m}x_i\longmapsto \prod\limits_{i\in\mathbb{Z}_m}\lambda_i(x_i).
\end{aligned}\end{equation}
\end{corollary}

\section{Odd periodic derived Hall algebras}
In this section, we assume that $m$ is an odd positive integer unless otherwise stated. We define another Hall algebra for $D_m(\A)$ without appending the $K$-elements in the basis elements.

\begin{definition}\label{jidef}
The Hall algebra $\mathcal {D}\mathcal {H}_m(\A)$, called the {\em $m$-periodic derived Hall algebra} of $\A$, is the $\mathbb{C}$-vector space with the basis $\{u_{\bigoplus\limits_{i\in\mathbb{Z}_m}M_i[i]}~|~[M_i]\in {\rm Iso}(\A)~\text{for~all~}i\in\mathbb{Z}_m\}$, and with the multiplication defined on basis elements by
{\begin{equation}\begin{split}&u_{\bigoplus\limits_{i\in\mathbb{Z}_m}A_i[i]}u_{\bigoplus\limits_{i\in\mathbb{Z}_m}B_i[i]}\\&=
v^{\zsum\lr{\sum\limits_{k=0}^{m-1}(-1)^k\hat{A}_{i+k},\hat{B}_i}}\sum\limits_{[I_i],[M_i]\in{\rm Iso}(\A), i\in\mathbb{Z}_m}
\prod\limits_{i\in\mathbb{Z}_m}\frac{H_{I_i[1]\oplus A_i,B_i\oplus I_{i-1}[-1]}^{M_i}}{a_{I_i}} u_{\bigoplus\limits_{i\in\mathbb{Z}_m}M_i[i]}.\end{split}\end{equation}}
\end{definition}
\begin{theorem}\label{jimain}
The $m$-periodic derived Hall algebra $\mathcal {D}\mathcal {H}_m(\A)$ is an associative algebra.
\end{theorem}
\begin{proof}
We only need to prove that
\begin{equation}\label{rjhl2}
(u_{\bigoplus\limits_{i\in\mathbb{Z}_m}A_i[i]}u_{\bigoplus\limits_{i\in\mathbb{Z}_m}B_i[i]})u_{\bigoplus\limits_{i\in\mathbb{Z}_m}C_i[i]}
=u_{\bigoplus\limits_{i\in\mathbb{Z}_m}A_i[i]}(u_{\bigoplus\limits_{i\in\mathbb{Z}_m}B_i[i]}u_{\bigoplus\limits_{i\in\mathbb{Z}_m}C_i[i]})\end{equation}
for any $A_i, B_i, C_i\in\A$ with $i\in\mathbb{Z}_m$.

On the one hand,
\begin{flalign*}
\mbox{LHS~~of}~~(\ref{rjhl2})=
\sum\limits_{[I_i],[X_i]\in{\rm Iso}(\A),i\in\mathbb{Z}_m}v^{a}\zprod\frac{H_{I_i[1]\oplus A_i,B_i\oplus I_{i-1}[-1]}^{X_i}}{a_{I_i}}u_{\osum X_i[i]}u_{\bigoplus\limits_{i\in\mathbb{Z}_m}C_i[i]},
\end{flalign*}
where $a=\zsum\lr{\sum\limits_{k=0}^{m-1}(-1)^k\hat{A}_{i+k},\hat{B}_i}.$
Thus,
\begin{equation}\label{tuizuobian}
\begin{split}
&\mbox{LHS~~of}~~(\ref{rjhl2})=\\
&\sum\limits_{[I_i],[X_i],[J_i],[M_i]\in{\rm Iso}(\A), i\in\mathbb{Z}_m}v^{b}\zprod\frac{H_{I_i[1]\oplus A_i,B_i\oplus I_{i-1}[-1]}^{X_i}}{a_{I_i}}
\zprod\frac{H_{J_i[1]\oplus X_i,C_i\oplus J_{i-1}[-1]}^{M_i}}{a_{J_i}} u_{\osum M_i[i]},
\end{split}\end{equation}
where $b=a+\zsum\lr{\sum\limits_{k=0}^{m-1}(-1)^k\hat{X}_{i+k},\hat{C}_i}.$

Using Lemma \ref{hallzhuan}(1), we obtain that
\begin{flalign*}
&\mbox{LHS~~of}~~(\ref{rjhl2})=
\sum\limits_{[I_i],[X_i],[J_i],[M_i]\in{\rm Iso}(\A),i\in\mathbb{Z}_m}v^{b}q^c\zprod\frac{a_{A_i}a_{B_i}a_{I_{i-1}}}{a_{X_i}}\zprod F_{A_i,I_{i-1}[-1],I_i[1],B_i}^{X_i}\\
&\quad\quad\quad\quad\quad\quad\quad\quad\quad\quad\quad \zprod\frac{a_{X_i}a_{C_i}a_{J_{i-1}}}{a_{M_i}}\zprod F_{X_i,J_{i-1}[-1],J_i[1],C_i}^{M_i}
u_{\osum M_i[i]},
\end{flalign*}
where $c=\zsum (\lr{\hat{I}_{i-1},\hat{I}_i}-\lr{\hat{A}_i,\hat{I}_i}-\lr{\hat{I}_{i-1},\hat{B}_i}+\lr{\hat{J}_{i-1},\hat{J}_i}-\lr{\hat{X}_i,\hat{J}_i}-\lr{\hat{J}_{i-1},\hat{C}_i})$.

Similar to the proof of Theorem \ref{mainresult0}, we can get that
\begin{equation}\label{zuobian2}
\begin{split}
&\mbox{LHS~~of}~~(\ref{rjhl2})=
\sum\limits_{[I_i],[M_i],[T_i],[U_i],[V_i],[W_i],[Z_i]\in{\rm Iso}(\A),i\in\mathbb{Z}_m}v^{b}q^{c-\sum\limits_{i\in\mathbb{Z}_m}\lr{\hat{U}_i,\hat{Z}_i}}\zprod\frac{a_{I_{i-1}}a_{T_i}a_{U_i}a_{V_i}a_{W_i}a_{Z_i}}{a_{M_i}}\\
&\quad\quad\quad\quad\quad\quad\quad\quad\zprod(F_{I_{i-1},U_{i},V_i}^{A_i}F_{W_i,Z_i,I_i}^{B_i}F_{T_i,U_{i+1},W_{i+1}}^{C_i}F_{V_i,Z_i,T_i}^{M_i}) u_{\osum M_i[i]}.
\end{split}
\end{equation}

On the other hand,
\begin{flalign*}
\mbox{RHS~~of}~~(\ref{rjhl2})=
\sum\limits_{[J_i],[Y_i]\in{\rm Iso}(\A),i\in\mathbb{Z}_m}v^{x}\zprod\frac{H_{J_i[1]\oplus B_i,C_i\oplus J_{i-1}[-1]}^{Y_i}}{a_{J_i}}u_{\bigoplus\limits_{i\in\mathbb{Z}_m}A_i[i]}u_{\osum Y_i[i]},
\end{flalign*}
where $x=\zsum\lr{\sum\limits_{k=0}^{m-1}(-1)^k\hat{B}_{i+k},\hat{C}_i}.$ Thus,
\begin{equation}\label{tuiyoubian}
\begin{split}
&\mbox{RHS~~of}~~(\ref{rjhl2})=\\
&\sum\limits_{[I_i],[Y_i],[J_i],[M_i]\in{\rm Iso}(\A),i\in\mathbb{Z}_m}v^{y}\zprod\frac{H_{I_i[1]\oplus A_i,Y_i\oplus I_{i-1}[-1]}^{M_i}}{a_{I_i}}
\zprod\frac{H_{J_i[1]\oplus B_i,C_i\oplus J_{i-1}[-1]}^{Y_i}}{a_{J_i}} u_{\osum M_i[i]},
\end{split}\end{equation}
where $y=x+\zsum\lr{\sum\limits_{k=0}^{m-1}(-1)^k\hat{A}_{i+k},\hat{Y}_i}.$

Using Lemma \ref{hallzhuan}(1), we obtain that
\begin{flalign*}
&\mbox{RHS~~of}~~(\ref{rjhl2})=
\sum\limits_{[I_i],[Y_i],[J_i],[M_i]\in{\rm Iso}(\A),i\in\mathbb{Z}_m}v^{y}q^z\zprod\frac{a_{A_i}a_{Y_i}a_{I_{i-1}}}{a_{M_i}}\zprod F_{A_i,I_{i-1}[-1],I_i[1],Y_i}^{M_i}\\
&\quad\quad\quad\quad\quad\quad\quad\quad\quad\quad\quad \zprod\frac{a_{B_i}a_{C_i}a_{J_{i-1}}}{a_{Y_i}}\zprod F_{B_i,J_{i-1}[-1],J_i[1],C_i}^{Y_i}
u_{\osum M_i[i]},
\end{flalign*}
where $z=\zsum (\lr{\hat{I}_{i-1},\hat{I}_i}-\lr{\hat{A}_i,\hat{I}_i}-\lr{\hat{I}_{i-1},\hat{Y}_i}+\lr{\hat{J}_{i-1},\hat{J}_i}-\lr{\hat{B}_i,\hat{J}_i}-\lr{\hat{J}_{i-1},\hat{C}_i})$.

Similar to the proof of Theorem \ref{mainresult0}, we can get that
\begin{equation}\label{youbian2}
\begin{split}
&\mbox{RHS~~of}~~(\ref{rjhl2})=
\sum\limits_{[J_i],[M_i],[N_i],[K_i],[P_i],[Q_i],[R_i]\in{\rm Iso}(\A),i\in\mathbb{Z}_m}v^{y}q^{z-\sum\limits_{i\in\mathbb{Z}_m}\lr{\hat{K}_i,\hat{R}_i}}\zprod\frac{a_{J_{i-1}}a_{N_i}a_{K_i}a_{P_i}a_{Q_i}a_{R_i}}{a_{M_i}}\\
&\quad\quad\quad\quad\quad\quad \zprod(F_{Q_{i-1},R_{i-1},N_i}^{A_i}F_{J_{i-1},K_i,Q_i}^{B_i}F_{P_i,R_{i},J_{i}}^{C_i}F_{N_i,K_i,P_i}^{M_i}) u_{\osum M_i[i]}.
\end{split}
\end{equation}

Replacing the notations $I_i, T_i, U_i, V_i, W_i, Z_i$ in (\ref{zuobian2}) by $Q_i, P_i, R_{i-1}, N_i, J_{i-1}, K_i$, respectively, for each $i\in\mathbb{Z}_m$,
we conclude that all terms in (\ref{zuobian2}) are the same as those in (\ref{youbian2}) except the exponents of $v$ and $q$.

After replacing the notations, similar to the proof of Theorem \ref{mainresult0}, we have that \begin{equation*}c=\sum\limits_{i\in\mathbb{Z}_m}(\lr{\hat{J}_{i-1},\hat{R}_i}-\lr{\hat{A}_i,\hat{J}_i}-\lr{\hat{B}_i,\hat{R}_i}+\lr{\hat{Q}_i,\hat{R}_i})+\text{Common~terms}~(\clubsuit)\end{equation*} and \begin{equation*}z=\sum\limits_{i\in\mathbb{Z}_m}(\lr{\hat{R}_{i-1},\hat{Q}_i}-\lr{\hat{R}_{i-1},\hat{B}_i}-\lr{\hat{Q}_{i-1},\hat{C}_i}+\lr{\hat{R}_{i-1},\hat{J}_{i-1}})+\text{Common~terms}~(\clubsuit),\end{equation*}
where $(\clubsuit)$ denotes the common terms of $c$ and $z$.

Using $\hat{X}_{i+k}=\hat{A}_{i+k}+\hat{B}_{i+k}-\hat{I}_{i+k}-\hat{I}_{i+k-1}$, $\hat{Y}_i=\hat{B}_i+\hat{C}_i-\hat{J}_i-\hat{J}_{i-1}$ and replacing the notations, we get that
\begin{flalign*}b=-\zsum\lr{\sum\limits_{k=0}^{m-1}(-1)^k(\hat{Q}_{i+k}+\hat{Q}_{i+k-1}),\hat{C}_i}+(\spadesuit)\end{flalign*}
and
\begin{flalign*}y=-\zsum\lr{\sum\limits_{k=0}^{m-1}(-1)^k\hat{A}_{i+k},\hat{J}_{i}+\hat{J}_{i-1}}+(\spadesuit),
\end{flalign*}
where \begin{flalign}\label{taohua}
(\spadesuit)=\zsum(\lr{\sum\limits_{k=0}^{m-1}(-1)^k\hat{A}_{i+k},\hat{B}_i}+\lr{\sum\limits_{k=0}^{m-1}(-1)^k\hat{A}_{i+k},\hat{C}_i}+\lr{\sum\limits_{k=0}^{m-1}(-1)^k\hat{B}_{i+k},\hat{C}_i}).
\end{flalign}
Since $m$ is odd, it is easy to see that $$\zsum\lr{\sum\limits_{k=0}^{m-1}(-1)^k(\hat{Q}_{i+k}+\hat{Q}_{i+k-1}),\hat{C}_i}=2\zsum\lr{\hat{Q}_{i-1},\hat{C}_i}$$
and \begin{flalign*}&\zsum\lr{\sum\limits_{k=0}^{m-1}(-1)^k\hat{A}_{i+k},\hat{J}_{i}+\hat{J}_{i-1}}\\
&=\zsum\lr{\sum\limits_{k=0}^{m-1}(-1)^k\hat{A}_{i+k},\hat{J}_{i}}+\zsum\lr{\sum\limits_{k=0}^{m-1}(-1)^k\hat{A}_{i+k},\hat{J}_{i-1}}\\
&=\zsum\lr{\sum\limits_{k=0}^{m-1}(-1)^k\hat{A}_{i+k},\hat{J}_{i}}+\zsum\lr{\sum\limits_{k=0}^{m-1}(-1)^k\hat{A}_{i+k+1},\hat{J}_{i}}\\
&=\zsum\lr{\sum\limits_{k=0}^{m-1}(-1)^k(\hat{A}_{i+k}+\hat{A}_{i+k+1}),\hat{J}_{i}}\\
&=2\zsum\lr{\hat{A}_i, \hat{J}_i}.\end{flalign*}
Thus, \begin{flalign}\label{by}
b=-2\zsum\lr{\hat{Q}_{i-1},\hat{C}_i}+(\spadesuit)~\text{and}~y=-2\zsum\lr{\hat{A}_i, \hat{J}_i}+(\spadesuit).
\end{flalign}
Hence, the exponent in (\ref{zuobian2}) is
$2c+b-2\sum\limits_{i\in\mathbb{Z}_m}\lr{\hat{R}_{i-1},\hat{K}_i}$ after we replace the notations $U_i,Z_i$ by $R_{i-1}, K_i$, respectively, which is equal to
\begin{equation}\label{jieshuzuobian2}
\begin{split}
\sum\limits_{i\in\mathbb{Z}_m}&(2\lr{\hat{J}_{i-1},\hat{R}_i}-2\lr{\hat{A}_i,\hat{J}_i}-2\lr{\hat{B}_i,\hat{R}_i}+2\lr{\hat{Q}_i,\hat{R}_i}-2\lr{\hat{Q}_{i-1},\hat{C}_i}\\
&-2\lr{\hat{R}_{i-1},\hat{B}_i-\hat{Q}_i-\hat{J}_{i-1}})+2(\clubsuit)+(\spadesuit)
\end{split}
\end{equation}
and the exponent in (\ref{youbian2}) is
$2z+y-2\sum\limits_{i\in\mathbb{Z}_m}\lr{\hat{K}_{i},\hat{R}_i}$, which is equal to
\begin{equation}\label{jieshuyoubian2}
\begin{split}
\sum\limits_{i\in\mathbb{Z}_m}&(2\lr{\hat{R}_{i-1},\hat{Q}_i}-2\lr{\hat{R}_{i-1},\hat{B}_i}-2\lr{\hat{Q}_{i-1},\hat{C}_i}+2\lr{\hat{R}_{i-1},\hat{J}_{i-1}}-2\lr{\hat{A}_i,\hat{J}_i}\\&-2\lr{\hat{B}_i-\hat{Q}_i-\hat{J}_{i-1},\hat{R}_i})
+2(\clubsuit)+(\spadesuit).
\end{split}
\end{equation}
It is easy to see that the formulas (\ref{jieshuzuobian2}) and (\ref{jieshuyoubian2}) are equal.
Hence, the exponents in (\ref{zuobian2}) and (\ref{youbian2}) are also the same.
Therefore, we complete the proof.
\end{proof}
\begin{remark}
In the proof Theorem \ref{jimain}, suppose that $m$ is even, then we can obtain that $y=b=0$ and Definition \ref{jidef} fails to provide an associative algebra.
We have not yet known whether one can get a definition of even periodic derived Hall algebra by amending the exponent of $v$ in Definition \ref{jidef}.
\end{remark}
\begin{corollary}
Let $m$ be an odd positive integer. For any $A_i,B_i,C_i,M_i\in\A$ with $i\in\mathbb{Z}_m$, we have that
\begin{flalign*}
&\sum\limits_{[I_i],[X_i],[J_i]\in{\rm Iso}(\A),i\in\mathbb{Z}_m}q^{-\sum\limits_{i\in\mathbb{Z}_m}\lr{\hat{I}_{i-1},\hat{C}_i}}\zprod(\frac{H_{I_i[1]\oplus A_i,B_i\oplus I_{i-1}[-1]}^{X_i}}{a_{I_i}}
\frac{H_{J_i[1]\oplus X_i,C_i\oplus J_{i-1}[-1]}^{M_i}}{a_{J_i}})\\
&=\sum\limits_{[I_i],[Y_i],[J_i]\in{\rm Iso}(\A),i\in\mathbb{Z}_m}q^{-\zsum\lr{\hat{A}_i, \hat{J}_i}}\zprod(\frac{H_{I_i[1]\oplus A_i,Y_i\oplus I_{i-1}[-1]}^{M_i}}{a_{I_i}}
\frac{H_{J_i[1]\oplus B_i,C_i\oplus J_{i-1}[-1]}^{Y_i}}{a_{J_i}}).
\end{flalign*}
\end{corollary}
\begin{proof}
By (\ref{by}), we have the exponents
\begin{flalign}
b=-2\zsum\lr{\hat{I}_{i-1},\hat{C}_i}+(\spadesuit)~\text{and}~y=-2\zsum\lr{\hat{A}_i, \hat{J}_i}+(\spadesuit)
\end{flalign}
in (\ref{tuizuobian}) and (\ref{tuiyoubian}), respectively, where $(\spadesuit)$ is the same as $(\ref{taohua})$.
Hence, by Theorem \ref{jimain}, we complete the proof.
\end{proof}

In what follows, we illustrate the relation between the $m$-periodic derived Hall algebra $\mathcal {D}\mathcal {H}_m(\A)$ in Definition \ref{jidef} and the Hall algebra for $D_m(\A)$ defined by Xu-Chen \cite{XuChen}. To this end, first of all, we give the following

\begin{proposition}\label{guanjmt}
For any objects $A_i, B_i, M_i\in\A$ with $i\in\mathbb{Z}_m$, we have that
\begin{equation}\label{guanjgs}
\begin{split}
&|\Hom_{D_m(\A)}(\bigoplus\limits_{i\in\mathbb{Z}_m}A_i[i],\bigoplus\limits_{i\in\mathbb{Z}_m}B_i[i+1])_{\bigoplus\limits_{i\in\mathbb{Z}_m}M_i[i+1]}|=\\&
\sum\limits_{[I_0],[I_1],\cdots,[I_{m-1}]\in{\rm Iso}(\A)}\zprod\frac{|\Hom_{D^b(\A)}(I_i[1]\oplus A_i,B_i[1]\oplus I_{i-1})_{M_i[1]}|}{a_{I_i}}.
\end{split}
\end{equation}
\end{proposition}
\begin{proof}
First of all, by \cite[Proposition 2.6]{XiaoXu}, we know that for any objects $X,Y,Z$ in a triangulated category $\mathcal {T}$, there exists a bijection
$$\frac{W_{\mathcal {T}}^{Y,Z,X}}{\Aut_\mathcal {T} (Y)}\longrightarrow \Hom_{\mathcal {T}}(Z,X)_{Y[1]},$$
where $W_{\mathcal {T}}^{Y,Z,X}:=\{(f,g,h)~|~Y\xrightarrow{f}Z\xrightarrow{g}X\xrightarrow{h}Y[1]~~\mbox{is~a~triangle~in}~~\mathcal {T}\}$ and the group action of $\Aut_\mathcal {T} (Y)$ on $W_{\mathcal {T}}^{Y,Z,X}$ is defined via the following commutative diagram
$$\xymatrix{Y\ar[r]^-{f}\ar[d]_-{\xi}&Z\ar[r]^-{g}\ar@{=}[d]&X\ar[r]^-{h}\ar@{=}[d]&Y[1]\ar[d]^-{\xi[1]}\\
Y\ar[r]^-{\bar{f}}&Z\ar[r]^-{g}&X\ar[r]^-{\bar{h}}&Y[1]}$$
for any $\xi\in\Aut_\mathcal {T} (Y)$, i.e., $\xi.(f,g,h)=(f\circ\xi^{-1},g,\xi[1]\circ h)$.

Using \cite[Lemma 3.1]{Zhang}, we can reduce the proof of the equation (\ref{guanjgs}) to prove that there exists a bijection between the sets
\begin{equation}\label{guidaojihe}
\frac{W_{D_m(\A)}^{\osum M_i[i],\osum A_i[i],\osum B_i[i+1]}}{\Aut_{D_m(\A)}(\osum M_i[i])}\end{equation}
and
\begin{equation}\label{bingjihe}
\bigcup\limits_{[I_0],[I_1],\cdots,[I_{m-1}]\in{\rm Iso}(\A)}\frac{\E_{I_0,I_{m-1}}^{B_0,M_0,A_0}\times\E_{I_{m-1},I_{m-2}}^{B_{m-1},M_{m-1},A_{m-1}}\times\cdots\times\E_{I_1,I_0}^{B_1,M_1,A_1}}{\Aut_{\A}(M_0,M_1,\cdots,M_{m-1} ,I_0,I_1,\cdots,I_{m-1})},\end{equation}
where for any objects $X_i\in\A$, $\E_{X_0,X_4}^{X_1,X_2,X_3}$ denotes the set $$\{(f_0,f_1,f_2,f_3)~|~0\xrightarrow{}X_0\xrightarrow{f_0}X_1\xrightarrow{f_1}X_2\xrightarrow{f_2}X_3\xrightarrow{f_3}X_4\xrightarrow{}0~~\mbox{is~an~exact~sequence~in}~~ \A\}$$ and $\Aut_{\A}(X_0,X_1,\cdots,X_n):=\Aut_{\A}(X_0)\times\Aut_{\A}(X_1)\times\cdots\times\Aut_{\A}(X_n)$, and the group action in (\ref{bingjihe}) is defined by the following commutative diagrams
$$\xymatrix{0\ar[r]&I_i\ar[r]\ar[d]_-{\xi_i}&B_i\ar[r]\ar@{=}[d]&M_i\ar[r]\ar[d]^-{\eta_i}&A_i\ar[r]\ar@{=}[d]&I_{i-1}\ar[r]\ar[d]^-{\xi_{i-1}}&0\\
0\ar[r]&I_i\ar[r]&B_i\ar[r]&M_i\ar[r]&A_i\ar[r]&I_{i-1}\ar[r]&0}$$
for any $\xi_i\in\Aut_{\A}(I_i), \eta_i\in\Aut_{\A}(M_i)$ and $i\in\mathbb{Z}_m$.

In fact, for any triangle in $D_m(\A)$ $$\osum \widetilde{M}_i[i]\longrightarrow \osum \widetilde{A}_i[i]\longrightarrow \osum \widetilde{B}_i[i+1]\longrightarrow \osum \widetilde{M}_i[i+1],$$
we have the exact sequences (\ref{cyclicexact}) and (\ref{fiveterm}).
Thus, we get a correspondence $\varphi$ from the set (\ref{guidaojihe}) to the set (\ref{bingjihe}). It is easy to see that $\varphi$ is well-defined and bijective.
Therefore, we complete the proof.
\end{proof}

\begin{lemma}\label{zhouqihom}
For any objects $A_i, B_i\in\A$ with $i\in\mathbb{Z}_m$, we have that
\begin{equation}
|\Hom_{D_m(\A)}(\osum A_i[i],\osum B_i[i])|=\zprod(|\Hom_{\A}(A_i, B_{i})|\cdot|\Ext_{\A}^1(A_i,B_{i+1})|).
\end{equation}
\end{lemma}
\begin{proof}
This can be easily proved by using (\ref{guidaodef}) and the heredity of $\A$.
\end{proof}

Now, let us recall the Hall algebra for $D_m(\A)$ defined by Xu-Chen \cite{XuChen}.
Given objects $M,N,X\in D_m(\A)$, set $$[M,N]:=\prod\limits_{i=1}^m|\Hom_{D_m(\A)}(M[i],N)|^{(-1)^i}.$$ According to \cite{XuChen}, for any objects $X,Y,L\in D_m(\A)$, we have that
$$\frac{|\Hom_{D_m(\A)}(L,X)_{Y[1]}|}{|{\rm Aut}_{D_m(\A)}(X)|}\cdot\sqrt{\frac{[L,X]}{[X,X]}}=
\frac{|\Hom_{D_m(\A)}(Y,L)_{X}|}{|{\rm Aut}_{D_m(\A)}(Y)|}\cdot\sqrt{\frac{[Y,L]}{[Y,Y]}}=:\mathscr{F}_{X,Y}^L.$$
\begin{definition}
The \emph{derived Hall algebra} $\mathscr{DH}_m(\A)$ of $D_m(\A)$ is the $\mathbb{C}$-vector space with the basis $\{\mu_{[X]}~|~[X]\in {\rm Iso}(D_m(\A))\}$, and with the multiplication defined by
\begin{equation}
\mu_{[X]} \mu_{[Y]}=\sum\limits_{[L]\in {\rm Iso}(D_m(\A))}\mathscr{F}_{X,Y}^L\mu_{[L]}.\end{equation}
\end{definition}
By \cite{XuChen}, the derived Hall algebra $\mathscr{DH}_m(\A)$ is an associative and unital algebra.
Similar to (\ref{drpgs1}), the odd periodic derived Hall number $\mathscr{F}_{X,Y}^L$ also has the derived Riedtmann--Peng formula (cf. \cite{XiaoXu2,Sheng})
$$\mathscr{F}_{X,Y}^L=\frac{|\Ext^1_{D_m(\A)}(X,Y)_{L}|}{|\Hom_{D_m(\A)}(X,Y)|}\cdot\frac{1}{\sqrt{[X,Y]}}\cdot\frac{|{\rm Aut}_{D_m(\A)}(L)|}{|{\rm Aut}_{D_m(\A)}(X)||{\rm Aut}_{D_m(\A)}(Y)|}\cdot\sqrt{\frac{[L,L]}{[X,X][Y,Y]}},$$
where $\Ext^1_{D_m(\A)}(X,Y)_{L}:=\Hom_{D_m(\A)}(X,Y[1])_{L[1]}$.

Define the \emph{dual derived Hall algebra} $\mathscr{DH}'_m(\A)$ of $D_m(\A)$ to be the $\mathbb{C}$-vector space with the basis $\{u_{[X]}~|~[X]\in {\rm Iso}(D_m(\A))\}$, and with the multiplication defined by
$$u_{[X]} u_{[Y]}=\sum\limits_{[L]\in {\rm Iso}(D_m(\A))}\mathscr{H}_{X,Y}^Lu_{[L]},$$ where
$$\mathscr{H}_{X,Y}^L=\frac{|\Ext^1_{D_m(\A)}(X,Y)_{L}|}{|\Hom_{D_m(\A)}(X,Y)|}\cdot\frac{1}{\sqrt{[X,Y]}}.$$

Clearly, the map $\rho:\mathscr{DH}'_m(\A)\rightarrow \mathscr{DH}_m(\A), u_{[X]}\mapsto\sqrt{[X,X]}\cdot |{\rm Aut}_{D_m(\A)}(X)|\cdot\mu_{[X]}$ gives an isomorphism of algebras.

\begin{theorem}\label{liangjitg}
The $m$-periodic derived Hall algebra $\mathcal {D}\mathcal {H}_m(\A)$ is the same as the dual derived Hall algebra $\mathscr{DH}'_m(\A)$, and thus it is isomorphic to the derived Hall algebra $\mathscr{DH}_m(\A)$.
\end{theorem}
\begin{proof}
For any objects $\osum A_i[i],\osum B_i[i],\osum M_i[i]\in D_m(\A)$,
\begin{equation}
[\osum A_i[i],\osum B_i[i]]=\prod\limits_{k=1}^m|\Hom_{D_m(\A)}(\osum A_i[i+k],\osum B_i[i])|^{(-1)^k}.
\end{equation}
Since \begin{equation*}\begin{split}&\osum A_i[i+k]=\\&A_0[k]\oplus A_1[k+1]\oplus\cdots\oplus A_{m-k-1}[m-1]\oplus A_{m-k}\oplus A_{m-k+1}[1]\oplus\cdots\oplus A_{m-1}[k-1]\\&
\cong A_{m-k}\oplus A_{m-k+1}[1]\oplus\cdots\oplus A_{m-1}[k-1]\oplus A_0[k]\oplus A_1[k+1]\oplus\cdots\oplus A_{m-k-1}[m-1],\end{split}\end{equation*}
by Lemma \ref{zhouqihom}, we obtain that
$$|\Hom_{D_m(\A)}(\osum A_i[i+k],\osum B_i[i])|=\zprod(|\Hom_{\A}(A_{i+m-k}, B_{i})|\cdot|\Ext_{\A}^1(A_{i+m-k},B_{i+1})|).$$
By definition, $$\mathscr{H}_{\osum A_i[i],\osum B_i[i]}^{\osum M_i[i]}=\frac{|\Ext^1_{D_m(\A)}(\osum A_i[i],\osum B_i[i])_{\osum M_i[i]}|}{|\Hom_{D_m(\A)}(\osum A_i[i],\osum B_i[i])|}\cdot\frac{1}{\sqrt{[\osum A_i[i],\osum B_i[i]]}}.$$
Noting that \begin{flalign*}
|\Hom_{D_m(\A)}(\osum A_i[i],\osum B_i[i])|&=\zprod(|\Hom_{\A}(A_i, B_{i})|\cdot|\Ext_{\A}^1(A_i,B_{i+1})|)\\
&=q^{\zsum{\rm dim}_k{\rm Ext}_{\A}^1(A_i,B_{i+1})}\cdot\zprod|\Hom_{\A}(A_i,B_i)|
\end{flalign*}
and
\begin{flalign*}
\sqrt{[\osum A_i[i],\osum B_i[i]]}=v^{\sum\limits_{k=1}^m(-1)^k(\zsum{\rm dim}_k{\rm Hom}_{\A}(A_{i+m-k}, B_{i})+\zsum{\rm dim}_k{\rm Ext}_{\A}^1(A_{i+m-k},B_{i+1}))},
\end{flalign*}
we get that
\begin{flalign*}
\mathscr{H}_{\osum A_i[i],\osum B_i[i]}^{\osum M_i[i]}=v^{-t_0}\cdot\frac{|\Hom_{D_m(\A)}(\bigoplus\limits_{i\in\mathbb{Z}_m}A_i[i],\bigoplus\limits_{i\in\mathbb{Z}_m}B_i[i+1])_{\bigoplus\limits_{i\in\mathbb{Z}_m}M_i[i+1]}|}
{\zprod|\Hom_{\A}(A_i,B_i)|},
\end{flalign*}
where \begin{flalign*}t_0&=\sum\limits_{k=1}^m(-1)^k(\zsum{\rm dim}_k{\rm Hom}_{\A}(A_{i+m-k}, B_{i})+\zsum{\rm dim}_k{\rm Ext}_{\A}^1(A_{i+m-k},B_{i+1}))\\&
\quad\quad\quad\quad\quad\quad\quad\quad\quad\quad+2\zsum{\rm dim}_k{\rm Ext}_{\A}^1(A_i,B_{i+1}).\end{flalign*}

Now let us calculate the exponent $-t_0$, it is equal to
\begin{flalign*}&\sum\limits_{j=0}^{m-1}(-1)^{j}(\zsum{\rm dim}_k{\rm Hom}_{\A}(A_{i+j}, B_{i})+\zsum{\rm dim}_k{\rm Ext}_{\A}^1(A_{i+j},B_{i+1}))
-2\zsum{\rm dim}_k{\rm Ext}_{\A}^1(A_i,B_{i+1})\\
&=\zsum{\rm dim}_k{\rm Hom}_{\A}(A_{i}, B_{i})+\zsum{\rm dim}_k{\rm Ext}_{\A}^1(A_{i},B_{i+1})-2\zsum{\rm dim}_k{\rm Ext}_{\A}^1(A_{i},B_{i+1})\\
&\quad-\zsum{\rm dim}_k{\rm Hom}_{\A}(A_{i+1}, B_{i})-\zsum{\rm dim}_k{\rm Ext}_{\A}^1(A_{i+1},B_{i+1})\\
&\quad+\zsum{\rm dim}_k{\rm Hom}_{\A}(A_{i+2}, B_{i})+\zsum{\rm dim}_k{\rm Ext}_{\A}^1(A_{i+2},B_{i+1})\\
&\quad-\cdots\\
&\quad-\zsum{\rm dim}_k{\rm Hom}_{\A}(A_{i+m-2}, B_{i})-\zsum{\rm dim}_k{\rm Ext}_{\A}^1(A_{i+m-2},B_{i+1})\\
&\quad+\zsum{\rm dim}_k{\rm Hom}_{\A}(A_{i+m-1}, B_{i})+\zsum{\rm dim}_k{\rm Ext}_{\A}^1(A_{i+m-1},B_{i+1})\\
&=\zsum(\lr{\hat{A}_i,\hat{B}_i}-\lr{\hat{A}_{i+1},\hat{B}_i}+\cdots-\lr{\hat{A}_{i+m-2},\hat{B}_i}+\lr{\hat{A}_{i+m-1},\hat{B}_i})\\
&=\zsum\lr{\sum\limits_{k=0}^{m-1}(-1)^k\hat{A}_{i+k},\hat{B}_i}.
\end{flalign*}
Using Proposition \ref{guanjmt}, we obtain that
\begin{flalign*}
&\mathscr{H}_{\osum A_i[i],\osum B_i[i]}^{\osum M_i[i]}=\\&v^{\zsum\lr{\sum\limits_{k=0}^{m-1}(-1)^k\hat{A}_{i+k},\hat{B}_i}}\sum\limits_{[I_0],[I_1],\cdots,[I_{m-1}]}\zprod\frac{|\Hom_{D^b(\A)}(I_i[1]\oplus A_i,B_i[1]\oplus I_{i-1})_{M_i[1]}|}{a_{I_i}|{\rm Hom}_{\A}(A_i,B_i)|}.
\end{flalign*}
By definition,
\begin{flalign*}H_{I_i[1]\oplus A_i,B_i\oplus I_{i-1}[-1]}^{M_i}
&=\frac{|{\rm Hom}_{D^b(\A)}(I_i[1]\oplus A_i,B_i[1]\oplus I_{i-1})_{M_i[1]}|}{|{\rm Hom}_{D^b(\A)}(I_i[1]\oplus A_i,B_i\oplus I_{i-1}[-1])|}\cdot\frac{1}{\{I_i[1]\oplus A_i,B_i\oplus I_{i-1}[-1]\}}\\
&=\frac{|{\rm Hom}_{D^b(\A)}(I_i[1]\oplus A_i,B_i[1]\oplus I_{i-1})_{M_i[1]}|}{|{\rm Hom}_{\A}(A_i,B_i)|}.\end{flalign*}
Thus, we have that
\begin{flalign*}
\mathscr{H}_{\osum A_i[i],\osum B_i[i]}^{\osum M_i[i]}=v^{\zsum\lr{\sum\limits_{k=0}^{m-1}(-1)^k\hat{A}_{i+k},\hat{B}_i}}\sum\limits_{[I_0],[I_1],\cdots,[I_{m-1}]\in{\rm Iso}(\A)}
\prod\limits_{i\in\mathbb{Z}_m}\frac{H_{I_i[1]\oplus A_i,B_i\oplus I_{i-1}[-1]}^{M_i}}{a_{I_i}}.\end{flalign*}
Hence, in the $m$-periodic derived Hall algebra $\mathcal {D}\mathcal {H}_m(\A)$,
\begin{equation*}u_{\bigoplus\limits_{i\in\mathbb{Z}_m}A_i[i]}u_{\bigoplus\limits_{i\in\mathbb{Z}_m}B_i[i]}=
\sum\limits_{[M_0],[M_1],[M_{m-1}]\in{\rm Iso}(\A)}
\mathscr{H}_{\osum A_i[i],\osum B_i[i]}^{\osum M_i[i]} u_{\bigoplus\limits_{i\in\mathbb{Z}_m}M_i[i]}.\end{equation*}
Therefore, we complete the proof.
\end{proof}

\section{Relations with Bridgeland's Hall algebras of $m$-periodic complexes}
In this section, we assume that $\A$ has enough projectives unless otherwise stated and relate the $m$-periodic extended derived Hall algebra $\mathcal {D}\mathcal {H}_m^{\rm e}(\A)$ with Bridgeland's Hall algebra of $m$-periodic complexes.
Bridgeland's Hall algebra of $2$-periodic complexes of $\A$ was introduced in \cite{Br}, and provides an entire realization of the corresponding quantum group. Inspired by the work of Bridgeland,  Chen and Deng \cite{ChenD} considered Bridgeland's Hall algebra $\mathcal{DH}_{\mathbb{Z}_m}(\A)$ of $m$-periodic complexes for any positive integer $m$.
For any objects $A,B,M,N\in\A$, set $$\gamma_{AB}^{MN}=\frac{a_Ma_N}{a_Aa_B}\sum\limits_{[I]\in{\rm Iso}(\A)}a_Ig_{L,M}^Ag_{N,I}^B.$$
From \cite{ZHC2}, we recall the algebra structure of Bridgeland's Hall algebra $\mathcal {D}\mathcal {H}_{\mathbb{Z}_m}(\A)$ for $m>2$ as follows:
\begin{proposition}\label{mBR} {\rm(\cite{ZHC2})}
If $m>2$. Then Bridgeland's Hall algebra
$\mathcal {D}\mathcal {H}_{\mathbb{Z}_m}(\A)$ is an associative and unital $\mathbb{C}$-algebra generated by the elements in $\{e_{A,i}~|~[A]\in{\rm Iso}(\A),~i\in \mathbb{Z}_m\}$ and $\{K_{\alpha,i}~|~\alpha\in K(\A),~i\in \mathbb{Z}_m\}$, and the following relations:
\begin{flalign}
&{K_{\alpha,i}} {K_{\beta,i}}={K_{\alpha+\beta,i}},~~~~
{K_{\alpha,i}} {K_{\beta,j}}=\begin{cases}
v^{( \alpha,\beta)}{K_{\beta,j}} {K_{\alpha,i}} \quad &\text{$i=j+1$},\label{5.1}\\
v^{-( \alpha,\beta)}{K_{\beta,j}} {K_{\alpha,i}} \quad &\text{$i=m-1+j$},\\
{K_{\beta,j}} {K_{\alpha,i}} & {\text{otherwise};}
\end{cases}\\
&{K_{\alpha,i}}e_{A,j}=\begin{cases}
v^{(\alpha,\hat{A})} e_{A,j}{K_{\alpha,i}} \quad &\text{$i=j$},\\
v^{-(\alpha,\hat{A})} e_{A,j}{K_{\alpha,i}} \quad &\text{$i=m-1+j$},\\
e_{A,j}{K_{\alpha,i}} & {\text{otherwise};}
\end{cases}\\
&e_{A,i}e_{B,i}=\sum\limits_{[M]}v^{\lr{\hat{A},\hat{B}}}g_{A,B}^Me_{M,i};\label{5.3}\\
&e_{A,i+1}e_{B,i}=\sum\limits_{[M],[N]}v^{\lr{\hat{A}-\hat{M},\hat{M}-\hat{N}}}\gamma_{AB}^{MN}K_{\hat{A}-\hat{M},i}e_{N,i}e_{M,i+1};\label{5.4}\\
&e_{A,i}e_{B,j}=e_{B,j}e_{A,i},~~i-j\neq0, 1~{\rm or}~m-1\label{5.5}.
\end{flalign}
\end{proposition}

For a general hereditary abelian category $\A$ which may not have enough projectives, inspired by the work \cite{Gorsky} of Gorsky, Lu and Peng \cite{LP} introduced the (twisted) semi-derived Ringel--Hall algebra of $2$-periodic complexes of $\A$ as a substitute of Bridgeland's Hall algebra of $2$-periodic complexes. Clearly, the construction of the (twisted) semi-derived Ringel--Hall algebra of $2$-periodic complexes can be easily generalised to $m$-periodic complexes (cf. \cite{Linji}).
According to \cite[Section 9.4]{Gorsky}, similarly, we have that Bridgeland's Hall algebra and the twisted semi-derived Ringel--Hall algebra of $m$-periodic complexes are isomorphic.
The twisted semi-derived Ringel--Hall algebra $^\imath\widetilde{\mathcal{H}}(\A)$ of $1$-periodic complexes, also called the $\imath$Hall algebra of $\A$, has been used in \cite{LWang} to give a realization of the corresponding $\imath$quantum group.

\begin{theorem}\label{briiso}
The $m$-periodic extended derived Hall algebra $\mathcal {D}\mathcal {H}_m^{\rm e}(\A)$ is isomorphic to Bridgeland's Hall algebra $\mathcal{DH}_{\mathbb{Z}_m}(\A)$ of $m$-periodic complexes.
\end{theorem}
\begin{proof}
$(i)$ For $m=1$, firstly, Bridgeland's Hall algebra $\mathcal{DH}_{\mathbb{Z}_1}(\A)$ is isomorphic to the $\imath$Hall algebra $^\imath\widetilde{\mathcal{H}}(\A)$.
Secondly, by \cite[Proposition 3.10]{LWang}, it is easy to see that $\mathcal {D}\mathcal {H}_1^{\rm e}(\A)$ is also isomorphic to $^\imath\widetilde{\mathcal{H}}(\A)$. Thus, we have the algebra isomorphism $\mathcal {D}\mathcal {H}_1^{\rm e}(\A)\cong\mathcal{DH}_{\mathbb{Z}_1}(\A)$.

$(ii)$ For $m=2$, by \cite{Yan}, Bridgeland's Hall algebra $\mathcal{DH}_{\mathbb{Z}_2}(\A)$ is isomorphic to the Drinfeld double Hall algebra of $\A$, which is isomorphic to $\mathcal {D}\mathcal {H}_2^{\rm e}(\A)$ by \cite{Zhang}. So, we have the algebra isomorphism $\mathcal {D}\mathcal {H}_2^{\rm e}(\A)\cong\mathcal{DH}_{\mathbb{Z}_2}(\A)$.

$(iii)$ For $m>2$, consider the map $$\varphi: \mathcal {D}\mathcal {H}_{\mathbb{Z}_m}(\A)\longrightarrow\mathcal {D}\mathcal {H}_m^{\rm e}(\A), e_{M,i}\mapsto \frac{1}{a_M} u_{M[i]}, K_{\alpha,i}\mapsto K_{\alpha,i}.$$ By the equations (\ref{yschengfa})-(\ref{kkjiaohuan}), it is easy to see that the relations (\ref{5.1})-(\ref{5.3}) are preserved under $\varphi$.
For any $i\in\mathbb{Z}_m$, by (\ref{yschengfa}), we have that
\begin{flalign*}
u_{A[i+1]}u_{B[i]}=\sum\limits_{[I_i],[M_i],[M_{i+1}]\in{\rm Iso}(\A)}v^{\lr{\hat{M}_i-\hat{M}_{i+1},\hat{I}_i}}\frac{H_{I_i[1],B}^{M_i}H_{A,I_i[-1]}^{M_{i+1}}}{a_{I_i}}u_{M_i[i]\oplus M_{i+1}[i+1]}K_{\hat{I}_i,i}.
\end{flalign*}
Also by (\ref{yschengfa}), we have that $u_{M_i[i]}u_{M_{i+1}[i+1]}=u_{M_i[i]\oplus M_{i+1}[i+1]}$.
By Lemmas \ref{hallzhuan} and \ref{jichuhallshu}, we get that
$$H_{I_i[1],B}^{M_i}=a_{I_i}F_{M_i,I_i}^B~\text{and}~H_{A,I_i[-1]}^{M_{i+1}}=a_{I_i}F_{I_i,M_{i+1}}^A.$$
By (\ref{kujiaohuan}), we have that
$$K_{\hat{I}_i,i}u_{M_i[i]\oplus M_{i+1}[i+1]}=v^{(\hat{I}_i,\hat{M}_i-\hat{M}_{i+1})}u_{M_i[i]\oplus M_{i+1}[i+1]}K_{\hat{I}_i,i}.$$
Thus, we obtain that
\begin{flalign*}
u_{A[i+1]}u_{B[i]}=\sum\limits_{[I_i],[M_i],[M_{i+1}]\in{\rm Iso}(\A)}v^{\lr{\hat{I}_i,\hat{M}_{i+1}-\hat{M}_{i}}}a_{I_i}F_{I_i,M_{i+1}}^AF_{M_i,I_i}^BK_{\hat{I}_i,i}u_{M_i[i]}u_{M_{i+1}[i+1]}.
\end{flalign*}
Noting that $\hat{I}_i=\hat{A}-\hat{M}_{i+1}$, we get that
\begin{flalign*}
u_{A[i+1]}u_{B[i]}=\sum\limits_{[M_i],[M_{i+1}]\in{\rm Iso}(\A)}v^{\lr{\hat{A}-\hat{M}_{i+1},\hat{M}_{i+1}-\hat{M}_{i}}}\frac{a_Aa_B}{a_{M_{i+1}}a_{M_{i}}}\gamma_{AB}^{M_{i+1}M_i}K_{\hat{A}-\hat{M}_{i+1},i}u_{M_i[i]}u_{M_{i+1}[i+1]}.
\end{flalign*}
Hence, the relation (\ref{5.4}) is preserved under $\varphi$.

By (\ref{yschengfa}), it is easy to see that for any $i,j\in\mathbb{Z}_m$ such that $i-j\neq0, 1~{\rm or}~m-1$,
$$u_{A[i]}u_{B[j]}=u_{B[j]}u_{A[i]}=u_{A[i]\oplus B[j]}.$$
So, the relation (\ref{5.5}) is also preserved under $\varphi$. Hence, the map $\varphi$ is an algebra homomorphism. By Proposition \ref{mjiaoji}, the map $\varphi$ is surjective.

By \cite[Proposition 4.4]{ChenD}, for each $i\in\mathbb{Z}_m$, the map $\kappa_i:{\mathcal {H}}_{\rm{tw}}^{\rm e}(\mathcal{A})\rightarrow\mathcal {D}\mathcal {H}_{\mathbb{Z}_m}(\A), u_{[M]}K_\alpha\mapsto a_Me_{M,i}K_{\alpha,i}$ is an embedding of algebras, and the map \begin{equation}\begin{aligned}
&\nu:\bigotimes\limits_{i\in\mathbb{Z}_m}{\mathcal {H}}_{\rm{tw}}^{\rm e}(\mathcal{A})\longrightarrow \mathcal {D}\mathcal {H}_{\mathbb{Z}_m}(\A),
&\bigotimes\limits_{i\in\mathbb{Z}_m}x_i\longmapsto \prod\limits_{i\in\mathbb{Z}_m}\kappa_i(x_i)
\end{aligned}\end{equation}
is an isomorphism of vector spaces. Then by the following commutative diagram
\begin{equation}
\xymatrix{\mathcal {D}\mathcal {H}_{\mathbb{Z}_m}(\A)\ar@{->>}[r]^-\varphi\ar@{->}[d]_-{\nu^{-1}}&\mathcal {D}\mathcal {H}_m^{\rm e}(\A)\\
\bigotimes\limits_{i\in\mathbb{Z}_m}{\mathcal {H}}_{\rm{tw}}^{\rm e}(\mathcal{A})\ar@{->}[ru]_-\mu&}
\end{equation}
we get that $\varphi$ is injective. Therefore, we complete the proof.
\end{proof}
\begin{remark}
By Theorem \ref{briiso}, we can say that the formula (\ref{exthalmul}) in the $m$-periodic extended derived Hall algebra $\mathcal {D}\mathcal {H}_m^{\rm e}(\A)$ provides a global, unified and explicit characterization for the algebra structure of Bridgeland's Hall algebra of periodic complexes.
\end{remark}


\end{document}